\documentclass{article}
\usepackage{lineno}

\usepackage[USenglish]{babel} 
\usepackage[utf8]{inputenc} 
\usepackage{comment}
\usepackage{authblk}

\usepackage{amsmath}
\usepackage{amsfonts}
\usepackage{amssymb}
\usepackage{xcolor}
\DeclareMathAlphabet{\pazocal}{OMS}{zplm}{m}{n}
\usepackage{cancel}

\usepackage[all,cmtip]{xy}

\usepackage{mathtools}
\mathtoolsset{showonlyrefs}
\usepackage{geometry}
\usepackage{indentfirst}
\usepackage{mathrsfs}
\usepackage{stmaryrd}

\usepackage{parskip}
\usepackage{indentfirst}
\usepackage{amsthm}
\usepackage{indentfirst}

\usepackage[normalem]{ulem}
\usepackage[all]{xy}
\usepackage{graphicx,color}
\usepackage{wrapfig}
\usepackage{hyperref} 
\usepackage{geometry}
\usepackage{graphicx,color}
\usepackage{wrapfig}
\usepackage{hyperref}

\usepackage{pb-diagram}
\usepackage{graphics}
\usepackage{tikz-cd}
\usepackage{tikz-qtree}
\usepackage{forest}
\usetikzlibrary{shapes.geometric,backgrounds}
\usepackage{pgf,tikz,pgfplots}\pgfplotsset{compat=1.14}
\usepackage{mathrsfs}
\usetikzlibrary{arrows}

\usepackage{mathabx}
\usepackage{lscape}
\usepackage{geometry}

\usepackage{enumitem}

\def\N{\mathbb{N}} 

\newtheorem{defi}{Definition}[section]
\newtheorem{remark}[defi]{Remark}

\newtheorem{thm}[defi]{Theorem}
\newtheorem{lemma}[defi]{Lemma}
\newtheorem{corollary}[defi]{Corollary}
\newtheorem{notation}[defi]{Notation}
\newtheorem{prop}[defi]{Proposition}
\newtheorem{question}[defi]{Question}

\newcommand{\dbl}{[\hspace{-0.2ex}[}
\newcommand{\dbr}{]\hspace{-0.2ex}]}
\newcommand{\db}[1]{\dbl {#1} \dbr}

\newcommand{\iso}{\cong}
\newcommand{\invlim}{\underleftarrow{\textnormal{lim}}\,}

\newcommand{\onto}{\twoheadrightarrow}

\newcommand{\Hom}{\textnormal{Hom}}

\newcommand{\tn}[1]{\textnormal{#1}}

\newcommand{\Br}{\textnormal{Br}}
\newcommand{\Ker}{\textnormal{Ker}}
\newcommand{\Aut}{\textnormal{Aut}}

\title{Block pro-fusion systems for profinite groups and blocks with infinite dihedral defect groups}
\author{Florian Eisele, Ricardo J.\ Franquiz Flores and John W.\ MacQuarrie}
\date{}

\begin{document}
	
	\maketitle
	\begin{abstract}
		We introduce block pro-fusion systems for blocks of profinite groups, prove a profinite version of Puig's structure theorem for nilpotent blocks, and use it to show that there is only one Morita equivalence class of blocks having the infinite dihedral pro-$2$ group as their defect group.
	\end{abstract}
	
    \section{Introduction}

    Classical block theory is an approach to the modular representation theory of finite groups.  If $k$ is an algebraically closed field of characteristic $p>0$ and $G$ is a finite group, one simply writes $kG$ as a product of indecomposable algebras -- the blocks -- and studies the representation theory of each block separately.  To each block $B$ is attached a $p$-subgroup $D$ of $G$, the defect group of $B$  (unique up to conjugation), which contains profound information about the representation theory of $B$.  Most blocks have wild representation type, but blocks whose defect group is cyclic have finite type and blocks whose defect group is dihedral, semi-dihedral or generalized quaternion have tame representation type. Further structural information is contained in the block fusion system of $B$, which is a fusion system on $D$. For instance, in the aforementioned tame cases, it determines the number of simple modules.   
    
    Long, complex projects due to Brauer, Dade, Green and others (in the finite type case) and Brauer, Donovan, Erdmann and others (in the tame cases) have resulted in classifications of the blocks of finite groups with finite or tame representation type. Going in a different direction, Puig developed a theory of nilpotent blocks, that is, blocks whose block fusion system is as small as it can be, which fully describes their structure.

    A block theory for profinite groups is in development, beginning with \cite{FranquizMacQuarrieBrauer}, where defect groups are defined and characterized, and Brauer's First Main theorem is proved.  In \cite{JRcyclicdef}, the blocks of profinite groups having (pro)cyclic defect group are classified using purely algebra theoretic methods: the classification of blocks with finite cyclic defect groups is invoked, and the (very few) possible inverse limits of such blocks are calculated using the classification. 
    The present paper is a contribution both to the general theory of blocks of profinite groups and the concrete problem of classifying blocks with certain defect groups. 
    
    As far as general theory is concerned, the aim of this paper is to construct a profinite version of block fusion systems, based on the definition of pro-fusion systems by Stancu and Symonds \cite{StancuSymonds}. This is not straight-forward since block fusion systems for blocks of finite groups do not fit into inverse systems in an obvious way. We manage to resolve this for blocks of countably based profinite groups, and obtain a natural definition of a \emph{block pro-fusion system} $\mathcal F_{(D,\hat e)}(G,b)$ via a suitable generalization Brauer pairs (see Definition~\ref{def via brauer pairs}). The relationship between $\mathcal F_{(D,\hat e)}(G,b)$ and the (ordinary) block fusion systems of the finite-dimensional blocks of which $k\db{G}b$ is the inverse limit, is rather intricate. On the one hand, $\mathcal F_{(D,\hat e)}(G,b)$ is an inverse limit of suitable quotients of finite block fusion systems arising from quotients of $k\db{G}b$ (see Theorem~\ref{thm fusion as inverse limit}). On the other hand the block fusion systems of the finite quotients of $k\db{G}b$ embed into suitable quotients of $\mathcal F_{(D,\hat e)}(G,b)$ (see Proposition~\ref{prop fusion embedding}). We can define nilpotent blocks of profinite groups in analogy with the finite case, and Puig's structure theorem still holds assuming the defect group is finitely generated as a pro-$p$ group. Of course the utility of fusion systems in the block theory of finite groups is not limited to this theorem, but other applications in the profinite case are beyond the scope of the present article.

    \begin{thm}[see Theorem~\ref{thm puig in body}]\label{thm puig}
        Let $G$ be a profinite group and let $B$ be a nilpotent block of $k\db{G}$ with topologically finitely generated defect group $D$. Then $B$ is Morita equivalent to $k\db{D}$.
    \end{thm}

    We then apply this to the ``pro-tame'' case, which was in fact our original motivation.  There is only one infinite profinite group that is the inverse limit of defect groups of tame blocks of finite groups: namely the infinite pro-$2$ dihedral group $D_{2^{\infty}}$.  Using the classification of blocks of tame type and extending the methods applied in \cite{JRcyclicdef}, one can show that there are at most three Morita equivalence classes of algebras which potentially contain blocks with defect group $D_{2^{\infty}}$.  We present these algebras in Section \ref{section limits of tame blocks}, as they are interesting in their own right.  However, we were not able to decide using these methods which of the three algebras appear as basic algebras of blocks.  Our main theorem is considerably stronger, and follows immediately from Theorem~\ref{thm puig} and the fact that $D_{2^\infty}$ does not support any non-trivial pro-fusion systems:

    \begin{thm}[see Corollary~\ref{Corollary Dinfty block is kDinfty in body}]\label{Theorem Dinfty block is kDinfty}
        If $B$ is a block of a profinite group whose defect group is the infinite pro-$2$ group $D_{2^{\infty}}$, then $B$ is Morita equivalent to $k\db{D_{2^{\infty}}}$.
    \end{thm}
    
    A result of the third author and Symonds \cite{MacQSymInPrep} says that a block of a profinite group $G$ with finite defect group is necessarily finite-dimensional, and is hence a block for some finite quotient of $G$.  Thus Theorem \ref{Theorem Dinfty block is kDinfty}, together with the known results for finite groups, yields a classification of all the blocks of a profinite group having finite or infinite dihedral defect group.  Of the auxiliary results collected in Section \ref{section prelims}, Proposition \ref{prop unique limit} may be interesting in its own right: it says that a bounded completed path algebra of a finite quiver is determined up to isomorphism by its continuous finite-dimensional quotients.

	
	
	\section{Preliminaries}\label{section prelims}
	
	\subsection{The pro-$2$ group $D_{2^{\infty}}$}

 The finite dihedral $2$-groups $D_{2^n} = \langle a,b\,|\,a^{2^n}, b^2, baba\rangle$ form an inverse system in the obvious way as $n$ varies, with inverse limit the infinite dihedral pro-$2$ group $D_{2^{\infty}}$.
 

 \begin{prop}\label{prop dihedral only invlim of tames}
     The only infinite inverse limit of a surjective inverse system of finite dihedral, semi-dihedral or generalized quaternion $2$-groups is $D_{2^{\infty}}$. 
 \end{prop}

 \begin{proof}
     This is very well-known, and follows easily from the fact that any proper non-abelian quotient of a group in the statement is dihedral.
 \end{proof}


    \subsection{Pseudocompact algebras and blocks}

        Throughout the text, $k$ is an algebraically closed field of characteristic $p$, treated where appropriate as a discrete topological ring.  
 
	\begin{defi}
        The topological $k$-algebra $A$ is \emph{pseudocompact} if it has a basis $B$ of open neighbourhoods of $0$ consisting of ideals of finite codimension, such that
    $$\bigcap_{I\in B}I = 0\quad\hbox{ and }\quad A = \invlim_{I\in B}A/I.$$
    Equivalently, a pseudocompact algebra is an inverse limit of discrete finite dimensional algebras.
 \end{defi}
	
	If $G = \invlim_N G/N$ is a profinite group (where $N$ runs through some cofinal set of open normal subgroups of $G$), then the group algebras $kG/N$ form an inverse system of finite dimensional algebras in the natural way, and hence $k\db{G} := \invlim_{N}kG/N$, the \emph{completed group algebra of $G$}, is a pseudocompact algebra. The algebra $k\db{G}$ is a product of indecomposable algebras called blocks, which are precisely the pseudocompact algebras $k\db{G}b$, where $b$ runs through the centrally primitive central idempotents $b$ of $k\db{G}$, which we refer to as block idempotents \cite[\S 4]{FranquizMacQuarrieBrauer}.

    As with finite groups, any block $k\db{G}b$ has a \emph{defect group}, a pro-$p$ subgroup of $G$ which can be defined in many equivalent ways, in perfect analogy with finite groups \cite[Theorem 5.18]{FranquizMacQuarrieBrauer}.  Defect groups exist and are unique up to conjugacy in $G$ \cite[Theorem 5.2 and Proposition 5.7]{FranquizMacQuarrieBrauer}.  A fundamental property of defect groups of profinite groups is that they are necessarily open (so of finite index) in any Sylow $p$-subgroup of $G$ that contains them \cite[Proposition 5.8]{FranquizMacQuarrieBrauer}.  
    
    Here we will be interested in blocks with countably based defect group $D$.
     A profinite group having a block with countably based defect group need not itself be countably based, but Corollary \ref{corol can take G countably based} will show that there is no loss of generality in assuming $G$ to be so.  The key to the proof is the following result from work in preparation by the third author and Symonds:

    \begin{prop}[\cite{MacQSymInPrep}]\label{prop john and peter in prep}
        Let $G$ be a profinite group with closed normal subgroup $N$, and denote by $e_N$ the block idempotent of the principal block of $k\db{N}$.  Then $e_N$ is central in $k\db{G}$ and the natural projection $\varphi_N : k\db{G}\to k\db{G/N}$ restricts to a surjection of algebras $k\db{G}e_N\to k\db{G/N}$.  This map is an isomorphism if, and only if, $N$ is a pro-$p'$ subgroup of $G$.
    \end{prop}

    \begin{corollary}\label{corol can take G countably based}
        Let $G$ be a profinite group and $B = k\db{G}b$ a block.  If the defect group of $B$ is countably based, then there is a countably based profinite group $H$ such that $B$ is isomorphic to a block of $k\db{H}$.
    \end{corollary}

    \begin{proof}
        Fix a $p$-Sylow subgroup $S$ of $G$ containing the defect group $D$ of $B$.  As noted above, $D$ is open in $S$, and hence if $D$ is countably based, then so is $S$.  So there are open subgroups $\{M_i\ : \  i\in \N\}$ of $S$ whose intersection is trivial.  For each $i$, let $N_i$ be an open normal subgroup of $G$ such that $N_i\cap S \subseteq M_i$.  Setting $N' = \bigcap_{i\in \N} N_i$, we have
        $$N'\cap S \subseteq \bigcap M_i = 1,$$
        so that (being normal) $N'$ does not intersect any $p$-Sylow, and hence is a pro-$p'$ subgroup of $G$.  Let $M$ be any open normal subgroup of $G$ for which $k\db{G}\to kG/M$ does not send $b$ to $0$, and set $N = M\cap N'$.  We can take $H = G/N$: with the notation of Proposition \ref{prop john and peter in prep}, we have $b\cdot e_N = b$ because $\varphi_N(be_N) = \varphi_N(b) \neq 0$, by the proposition.  Hence $k\db{G}b$ is a direct summand of $k\db{G}e_N \iso k\db{G/N}$, again by the proposition.
    \end{proof}

    Returning briefly to general algebras, the \emph{Jacobson radical} $J(A)$ of a pseudocompact algebra $A$ is the intersection of its maximal closed left ideals.  It is a closed two sided ideal and coincides with the Jacobson radical of $A$ considered as an abstract (meaning no topology) algebra {\cite[p.444]{Bru}, \cite[Proposition 3.2]{IusenkoMacQuarrieSemisimple}}.  The algebra $A/J(A)$ is (topologically) separable and if it is finite dimensional, as will be the case with the algebras we consider here, then it is separable in the usual sense.  For any $n>1$, we define inductively $J^n(A)$ to be the closed submodule of $A$ generated by $J(A)\cdot J^{n-1}(A)$.  We thus obtain a descending chain 
    $$\cdots\subseteq J^2(A)\subseteq J(A)\subseteq A$$
    of closed left ideals of $A$ whose intersection is $0$.

    \subsection{Morita equivalence}
    
    A pseudocompact algebra $A$ is \emph{basic} if its simple modules have dimension $1$, or equivalently if $A/J(A) \iso \prod_{i\in I} k$ for some indexing set $I$ \cite[Corollary 5.5]{simson2001}.  Any pseudocompact algebra is Morita equivalent to a basic pseudocompact algebra (\cite{G} or \cite[Proposition 5.6]{simson2001}).  We show here that Morita equivalence behaves well with respect to inverse systems of blocks.

    Let $B$ be a block of the profinite group $G$ having a finite number $n$ of simple modules, and let $P_1, \hdots, P_n$ be a complete set of representatives of the isomorphism classes of the indecomposable projective $B$-modules.  Define $P = \prod_{i=1}^nP_i$.  By general Morita theory 
and \cite[Lemma 2.3]{PPJ}, the algebra $A := \tn{End}_B(P)$ is pseudocompact, basic and Morita equivalent to $B$.

Let $\mathcal{N}$ denote the cofinal set of open normal subgroups of $G$ that act trivially on every simple $B$-module.  As in \cite[\S2.3]{FranquizMacQuarrieBrauer}, given a $k\db{G}$-module $U$, we define the module of $N$-coinvariants $U_N=U/I_NU$, where $I_N$ denotes the augmentation ideal of $k\db{N}$. The same applies to blocks by \cite[Remark~2.8]{FranquizMacQuarrieBrauer}: $k\db{G}_N$ is canonically isomorphic to $k\db{G/N}$ as pseudocompact algebras and $B_N$ is a direct factor of $k\db{G/N}$. Thus for any $N\in \mathcal{N}$ the algebra $B_N$ has $n$ simple modules,
 $(P_i)_N$ is non-zero, indecomposable, and not isomorphic to $(P_j)_N$ for any $j\neq i$, and hence $A_N := \tn{End}_{B_N}(P_N)$ is also basic.

\begin{prop}\label{prop inverse limit basic}
	The $A_N$ form a surjective inverse system of algebras and algebra homomorphisms, with inverse limit $A$.
\end{prop}

\begin{proof}
	It is routine to check that when $N\leqslant M$, the maps $\psi_{MN} : A_N\to A_M$ sending the endomorphism $\gamma$ of $P_N$ to the endomorphism $\gamma_M$ of $P_M$ yield an inverse system of algebras, with inverse limit $A$.  It remains to check surjectivity.  Given $\delta : P_M\to P_M$ in $A_M$ we have the following diagram 
	$$\xymatrix{P_N\ar[d]_{\varphi_{MN}} & P_N\ar[d]^{\varphi_{MN}} \\ 
		P_M\ar[r]_{\delta} & P_M}$$
	Regarding this as a diagram of $kG/N$-modules, it can be completed to a commutative square via $\delta' : P_N\to P_N$, by the projectivity of (the left hand) $P_N$.  Now $\psi_{MN}(\delta') = \delta$ and so the maps of the inverse system are surjective.
\end{proof}

It follows from the above proposition that if $B$ is a block having finitely many simple modules, then the basic algebra $A$ Morita equivalent to $B$ is the inverse limit of a surjective inverse system of basic algebras $A_N$, with $A_N$ Morita equivalent to $B_N$. 

    \subsection{Completed path algebras}\label{section completed path algebras}
    
    Any basic pseudocompact algebra can be described combinatorially, but to simplify the conversation, we restrict to the class of algebras that will interest us in this article: namely, those basic pseudocompact algebras $A$ for which $J^2(A)$ has finite codimension in $A$.

    A finite quiver $Q$ is simply a finite directed graph, with multiple edges and loops allowed.  A path of length $n$ in $Q$ ($n\geqslant 0$) is a sequence of $n$ composable arrows of $Q$.  There is a path $e_i$ of length $0$ at each vertex $i$ of $Q$.   For each $n$, let $kQ_n$ be the vector space with basis the paths of length $n$, and define the \emph{completed path algebra}
    $$k\db{Q} := \prod_{n\geqslant 0}kQ_n.$$
    The only difference between the completed and the usual path algebra is that we take the product rather than the sum. This is a basic pseudocompact algebra, with multiplication of paths defined in the obvious way: the product of two paths is the concatenation when they are composable, or $0$ otherwise.  We adopt the convention that paths are composed from right to left, so that for example if $Q$ is the quiver
    $$ \begin{tikzcd}
			2  
			& 1 \arrow[l,swap,"b"]  & 0 \arrow[l,swap,"a"] 
		\end{tikzcd}$$
    then $k\db{Q}$ ($= kQ$ in this example) has basis $\{e_0, e_1, e_2, a,b,ba\}$, and some examples of multiplication are
    $$b\cdot a = ba, a\cdot b = 0.$$
    For any $s\geqslant 1$ we have 
    $$J^s(k\db{Q}) = \prod_{n\geqslant s}kQ_n.$$
    A \emph{relation ideal} of $k\db{Q}$ is a closed ideal $I$ of $k\db{Q}$ contained in $J^2(k\db{Q})$, while an \emph{admissible ideal} is an ideal $I$ of $k\db{Q}$ such that $J^n(k\db{Q})\subseteq I \subseteq J^2(k\db{Q})$ for some $n\geqslant 2$.  A relation ideal is admissible if, and only if, $k\db{Q}/I$ is finite dimensional \cite[Proposition 5.3]{JK}.  Every basic pseudocompact algebra such that $J^2(A)$ has finite codimension in $A$ is isomorphic to an algebra of the form $k\db{Q}/I$, where $Q$ is a finite quiver and $I$ is a relation ideal of $k\db{Q}$ \cite[Chapter 6, \S 6]{Kra}. 

The following proposition is quite general and may be useful in other contexts.  
In our intended application, we will obtain a surjective inverse system 
$$\cdots \to k\db{Q}/I_3 \to k\db{Q}/I_2 \to k\db{Q}/I_1$$
where the ideals $I_n$ form a descending chain.  But we will not have control over the maps, so we must justify that the inverse limit is the ``obvious'' algebra $k\db{Q}/\bigcap_n I_n$:

\begin{prop}\label{prop unique limit}
    Let $Q$ be a finite quiver and $I_1 \supseteq I_2 \supseteq \hdots$ a chain of closed relation ideals of $k\db{Q}$, and set $I = \bigcap_{n\in \mathbb{N}}I_n$.  For each $s\in \mathbb{N}$, write $J^s = J^s(k\db{Q})$.
    For each $n\in \mathbb{N}$, let $\rho_{n,n+1} : k\db{Q}/I_{n+1}\to k\db{Q}/I_n$ be a surjective algebra homomorphism, and whenever $m\leqslant n$ define 
    $$\rho_{mn} := \rho_{m,m+1}\rho_{m+1,m+2}\hdots \rho_{n-1,n},$$
    so that $\{k\db{Q}/I_n, \rho_{mn}\}$ is an inverse system of algebras.  
    Then
    $$\invlim_{n\in \mathbb{N}}\{k\db{Q}/I_n, \rho_{mn}\}\iso k\db{Q}/I.$$
\end{prop}

\begin{proof}
    In order to avoid confusing indices, we introduce the following abuses of notation.  Firstly, whenever $m\leqslant n$ and $L$ is an ideal of $k\db{Q}$, we denote by $\pi_{mn} : k\db{Q}/(L+J^n)\to k\db{Q}/(L+J^m)$ the canonical projection.  Secondly, given ideals $L,L'$ of $k\db{Q}$ and a surjective algebra homomorphism $\gamma : k\db{Q}/L\to k\db{Q}/L'$, we denote also by $\gamma$ the induced homomorphism $k\db{Q}/(L+J^n)\to k\db{Q}/(L'+J^n)$.  Consider the following diagram:
    $$\xymatrix{ 
    & \ar[d] & \ar[d] & \ar[d] \\
\ar[r] & k\db{Q}/(I_3+J^3) \ar[r]_{\pi_{23}}\ar[d]^{\rho_{23}} & k\db{Q}/(I_3+J^2) \ar[r]_{\pi_{12}}\ar[d]^{\rho_{23}} & k\db{Q}/(I_3+J^1)\ar[d]^{\rho_{23}}  \\
\ar[r] & k\db{Q}/(I_2+J^3) \ar[r]_{\pi_{23}}\ar[d]^{\rho_{12}} & k\db{Q}/(I_2+J^2) \ar[r]_{\pi_{12}}\ar[d]^{\rho_{12}} & k\db{Q}/(I_2+J^1)\ar[d]^{\rho_{12}} \\
\ar[r] & k\db{Q}/(I_1+J^3) \ar[r]_{\pi_{23}} & k\db{Q}/(I_1+J^2) \ar[r]_{\pi_{12}} & k\db{Q}/(I_1+J^1) 
    }$$
Note that the squares commute.  For each fixed $n\in \mathbb{N}$, the $n$th row $\{k\db{Q}/(I_n + J^s), \pi_{st}\}$ is an inverse system, with inverse limit $k\db{Q}/I_n$.  

For each fixed $s$, the $s$th column is an inverse system $\{k\db{Q}/(I_n + J^s), \rho_{mn}\}$, whose limit we claim is $k\db{Q}/(I+J^s)$: the algebras $k\db{Q}/(I_n+ J^s)$ are quotients of the finite dimensional algebra $k\db{Q}/J^s$, so there must be $n_0\in \mathbb{N}$ for which $\rho_{mn} : k\db{Q}/(I_n+ J^s)\to k\db{Q}/(I_m + J^s)$ is an isomorphism whenever $n\geqslant m\geqslant n_0$.  Working in the inverse system of $n\geqslant n_0$, we define for each $n$ the map
$$\theta_n : k\db{Q}/(I+J^s)\xrightarrow{\pi_{n_0}} k\db{Q}/(I_{n_0} + J^s) \xrightarrow{{\rho_{n_0n}}^{-1}}k\db{Q}/(I_n + J^s).$$
The $\theta_n$ yield a surjective map of inverse systems $\{k\db{Q}/(I+J^s), \tn{id}\}\to \{k\db{Q}/(I_n + J^s), \rho_{mn}\}$ and hence a surjective algebra homomorphism
$$k\db{Q}/(I+ J^s)\to \invlim_n k\db{Q}/(I_n + J^s),$$
which is an isomorphism because, since $J^s$ has finite codimension in $k\db{Q}$, $I_n + J^s = I + J^s$ for sufficiently large $n$.

Now, because the squares commute, the vertical maps yield a map of inverse systems between any two adjacent rows, and the horizontal maps yield a map of inverse systems between any two adjacent columns.  Passing to the limits, we thus obtain
$$\xymatrix{
&\ar[r]&  k\db{Q}/(I+ J^3)\ar[r]^{\pi_{23}} & k\db{Q}/(I+ J^2)\ar[r]^{\pi_{12}} & k\db{Q}/(I+ J^1)
 & \\
   \ar[d] & & \vdots\ar[d] & \vdots\ar[d] & \vdots\ar[d] \\
k\db{Q}/I_3\ar[d]_{\rho_{23}} &\cdots\ar[r] & k\db{Q}/(I_3+J^3) \ar[r]_{\pi_{23}}\ar[d]^{\rho_{23}} & k\db{Q}/(I_3+J^2) \ar[r]_{\pi_{12}}\ar[d]^{\rho_{23}} & k\db{Q}/(I_3+J^1)\ar[d]^{\rho_{23}}  \\
k\db{Q}/I_2\ar[d]_{\rho_{12}} & \cdots\ar[r] & k\db{Q}/(I_2+J^3) \ar[r]_{\pi_{23}}\ar[d]^{\rho_{12}} & k\db{Q}/(I_2+J^2) \ar[r]_{\pi_{12}}\ar[d]^{\rho_{12}} & k\db{Q}/(I_2+J^1)\ar[d]^{\rho_{12}} \\
 k\db{Q}/I_1 & \cdots\ar[r] & k\db{Q}/(I_1+J^3) \ar[r]_{\pi_{23}} & k\db{Q}/(I_1+J^2) \ar[r]_{\pi_{12}} & k\db{Q}/(I_1+J^1) 
    }$$
By \cite[Proposition 2.12.1]{Borceux}, 
the inverse limit $\invlim_n \{k\db{Q}/I_n, \rho_{mn}\}$ of the left most vertical inverse system is isomorphic to the inverse limit of the upper horizontal inverse system, which is $k\db{Q}/I$. 
\end{proof}

\begin{remark}
    If one is interested in bounded completed path algebras of possibly infinite quivers, the proof of the above proposition allows the following generalization: if $Q$ is a quiver, $I_1 \supseteq I_2 \supseteq \hdots$ is a chain of relation ideals of $k\db{Q}$, and $\{k\db{Q}/I_n, \rho_{mn}\}$ is a surjective inverse system with the property that for every $s$, the maps of the induced inverse system $\{k\db{Q}/(I_n+ J^s), \rho_{mn}\}$ are eventually isomorphisms, then $\invlim_{n\in \mathbb{N}}\{k\db{Q}/I_n, \rho_{mn}\} \iso k\db{Q}/\bigcap I_n$.
\end{remark}

    \subsection{Fusion and pro-fusion systems}

    In this section we will provide the necessary background on block fusion systems and pro-fusion systems. Recall that
    a \emph{fusion system} on a finite $p$-group $P$ is a finite category $\mathcal F$ whose objects are the subgroups of $P$ and whose sets of homomorphisms $\Hom_{\mathcal F}(R,S)$, for any $R,S\leq P$, consist of injective group homomorphisms from $R$ into $S$ such that certain axioms are satisfied. If the category satisfies a further set of axioms it is called a \emph{saturated} fusion system (note that Linckelmann \cite{LinckelmannVolII} includes the saturation axioms in his definition of a fusion system, while many other authors keep the notions separate). We will not need the axioms explicitly, since all fusion systems in the present paper will come from blocks of finite groups, which are known a priori to be saturated fusion systems. The reader may wish to refer to \cite{LinckelmannVolII, LinckelmannFusionSystems, CravenFusionSystems} for comprehensive surveys of the theory. 

    A morphism between a fusion system $\mathcal F$ on $P$ and a fusion system $\mathcal F'$ on $Q$, where $P$ and $Q$ are finite $p$-groups, is given by a pair $(\alpha, \Phi)$, where $\alpha:\ P \longrightarrow Q$ is a group homomorphism and $\Phi:\ \mathcal F \longrightarrow \mathcal F'$  is a functor such that 
    \begin{enumerate}
        \item $\alpha(R)=\Phi(R)$ for all $R\leq P$, and
        \item $\Phi(\varphi)\circ \alpha=\alpha\circ \varphi$ for all $\varphi\in\Hom_{\mathcal F}(R,S)$, where $R,S\leq P$.
    \end{enumerate}
    The functor $\Phi$ is determined by $\alpha$, so we can think of morphisms between fusion systems as group homomorphisms between the underlying $p$-groups. But not all group homomorphisms give rise to morphisms of fusion systems. 
    Having defined morphisms of fusion systems, we can now think of fusion systems on finite $p$-groups as a category. In particular, it is now clear when two fusion systems are isomorphic. It is also clear what is meant by an inverse system of fusion systems, which will be important when defining pro-fusion systems later. 
    
    \paragraph{Block fusion systems.}
    The archetypical example of a fusion system is the category $\mathcal F_P(G)$, where $G$ is a finite group and $P$ is a (fixed) Sylow $p$-subgroup of $G$. The objects of $\mathcal F_P(G)$ are the subgroups of $P$ and the morphisms are group homomorphisms induced by conjugation by elements of $G$ followed by inclusions.
    
    Block fusion systems are a slight modification of this construction. The crucial ingredient in their definition are \emph{Brauer~pairs}. Their definition and an outline summary of the associated theory is given below.
    \begin{defi}[cf. {\cite[Definition 6.3.1]{LinckelmannVolII}}]
        Let $G$ be a finite group. A \emph{Brauer pair} for $kG$ is a pair $(P,e)$, where $P$ is a $p$-subgroup of $G$ and $e$ is a primitive idempotent of $Z(kC_G(P))$.
    \end{defi}
    We will also need the \emph{Brauer map}, which plays a role in the definiton of the relation ``$\leq$'' on Brauer pairs. 
    \begin{defi}
        Let $G$ be a finite group and let $P\unlhd Q \leq G$ be two $p$-subgroups. 
        The \emph{Brauer map} $\Br_P$ is the linear projection
        \begin{equation}
            \Br_Q: Z(kC_G(P))^Q \longrightarrow Z(kC_G(Q)): \sum_{g\in C_G(P)} a_g g \mapsto \sum_{g\in C_G(Q)} a_g g.\label{eqn def Brauer map}
        \end{equation}
    \end{defi}
    The theory of Brauer pairs and their relationship to blocks, their defect groups and fusion systems is explained in detail in \cite[Chapter 6]{LinckelmannVolII}, and we will refer to this reference for all facts we will be using. The main idea is that the Brauer pairs for $kG$ form a partially ordered set, with an obvious $G$-action preserving the partial order. 
    \begin{prop}[{cf. \cite[Proposition 6.3.4]{LinckelmannVolII}}]\label{prop partial order}
        Let $G$ be a finite group, and let $(Q,f)$ be a Brauer pair for $kG$. If $P$ is a normal subgroup of $Q$ then there exists a \emph{unique} $Q$-stable block idempotent $e\in Z(kC_G(P))$ such that $\Br_Q(e)f=f$. 
    \end{prop} 

    \begin{defi}[Partial order]
        In the situation of Proposition~\ref{prop partial order} we declare $(P,e)\leq (Q,f)$. The transitive closure of this relation defines a partial order ``$\leq$'' on all Brauer pairs for $kG$.
    \end{defi}

    \begin{prop}[{cf. \cite[Theorem 6.3.3]{LinckelmannVolII}}]
        Let $(Q,f)$ be a Brauer pair for $kG$ and let $P$ be a subgroup of $Q$. Then there is a unique idempotent $e$ such that $(P,e)\leq (Q,f)$.
    \end{prop}
    
    The blocks of $kG$ correspond to the Brauer pairs of the form $(1, b)$, and these are exactly the minimal Brauer pairs with respect to ``$\leq$''. 
    Given $(1, b)$, all $(D, e)$ which are maximal with respect to the property $(1, b)\leq (D, e)$ are $G$-conjugate, and the $p$-subgroups $D$ occurring in such maximal Brauer pairs are exactly the defect groups of the block corresponding to $b$. And once we fix such a maximal Brauer pair $(D, e)$, the poset of all Brauer pairs $\leq (D, e)$ is canonically identified with the poset of subgroups of $D$ with respect to inclusion, that is, any $Q\leq D$ fits into a unique Brauer pair $(Q,e_Q)$ with $(Q,e_Q)\leq (D,e)$. This leads to the definition of block fusion systems which we will generalize to the profinite setting.

    \begin{defi}
        Let $kGb$ be a block and let $(D,e)$ be a maximal Brauer pair with $(1,b)\leq (D,e)$. For any $Q\leq D$ let $e_Q\in Z(kC_G(Q))$ denote the unique block idempotent such that $(Q,e_Q)\leq (D,e)$. Then we define the block fusion system $\mathcal F=\mathcal F_{(D,e)}(G,b)$ as follows:
        \begin{enumerate}
            \item The objects are the subgroups of $D$, and
            \item for $P,Q\leq D$ we define $\Hom_{\mathcal F}(P,Q)$ to consist of all homomorphisms $\varphi:\ P \longrightarrow Q$ for which there exists a $g\in G$ such that
            \begin{enumerate}
                \item $\varphi(x) = x^g$ for all $x\in P$, and
                \item $(P^g, e_P^g)\leq (Q, e_Q)$.
            \end{enumerate}
        \end{enumerate}
    \end{defi}
    This is clearly a subcategory of $\mathcal F_D(G)$, but $D$ need not be a Sylow $p$-subgroup of $G$, so $\mathcal F_D(G)$ may fail to be a saturated fusion system. The category $\mathcal F_{(D,e)}(G,b)$ itself is known to always be a saturated fusion system, and we will not need to know much else about it. 
    Note that $\mathcal F_{(D,e)}(G,b)$ obviously depends on $(D,e)$, but since all admissible choices of $(D,e)$ are conjugate in $G$, the isomorphism type of  $\mathcal F_{(D,e)}(G,b)$ really only depends on the block $kGb$.

    \paragraph{Quotients.} 
    Morphisms between fusion systems were defined above, and it is fairly obvious how embeddings work. For instance, if $D$ is a Sylow $p$-subgroup of $G$, then $\mathcal F_{(D,e)}(G,b)$ embeds into  $\mathcal F_D(G)$. Since we will need to construct inverse systems of fusion systems, we will need quotients as well. The conditions for forming quotients given in \cite{LinckelmannVolII} are too restrictive, so we follow \cite{StancuSymonds} instead. 

    \begin{defi}
        Let $\mathcal F$ be a fusion system on the finite $p$-group $P$. A subgroup $S\leq P$ is \emph{strongly closed} if $\varphi(Q)\leq S$ for all $Q\leq S$ and all $\varphi\in\Hom_{\mathcal F}(Q, P)$.
    \end{defi}
    Note that in the block fusion system  $\mathcal F_{(D,e)}(G,b)$ the subgroup $N\cap D$ is strongly closed for any normal subgroup $N$ of $G$. The same is true in $\mathcal F_D(G)$ if $D$ is a Sylow $p$-subgroup of $G$. There may be more strongly closed subgroups in the case of block fusion systems, but for our purposes only the ones of the form $N\cap D$ will be needed. The important feature of strongly closed subgroups is that we can take quotients by them.
    \begin{defi}
        Let $\mathcal F$ be a fusion system on the finite $p$-group $P$ and let $S$ be a strongly closed subgroup of $P$. Then we can define a fusion system $\mathcal F/S$ on $P/S$ by letting $\Hom_{\mathcal F/S}(Q/S, R/S)$ for $Q,R\leq P$ with $S\leq Q,R$ be the image of $\Hom_{\mathcal F}(Q,R)$ under the natural map.
    \end{defi}
    
    It is not a priori clear that there is a morphism $\mathcal F\longrightarrow \mathcal F/S$, since the map from $\Hom_{\mathcal F}(Q,R)$ to $\Hom_{\mathcal F/S}(QS/S, RS/S)$ may be ill-defined when $S$ is not contained in $Q$ and/or $R$. However, for saturated fusion systems it turns out that there is indeed a morphism.
    \begin{prop}[{cf. \cite[Corollary 2.6]{StancuSymonds}}]\label{prop factor out strongly closed}
        Let $\mathcal F$ be a saturated fusion system on the finite $p$-group $P$ and let $S$ be a strongly closed subgroup of $P$. Then $\mathcal F/S$ is again a saturated fusion system and the natural maps induce a morphism $\mathcal F \longrightarrow \mathcal F/S$.
    \end{prop}
    The crucial ingredient for this is the following non-trivial fact: if $Q$ and $R$ are arbitrary subgroups of $P$, then for any $\varphi\in\Hom_{\mathcal F}(Q,R)$ there is a $\tilde\varphi \in\Hom_{\mathcal F}(QS,RS)$ which agrees with $\varphi$ modulo $S$ (but may, of course, fail to restrict to $\varphi$). See \cite[Theorem 2.5]{StancuSymonds}, or \cite[Proposition 6.3]{PuigFrobeniusCategories} for the original assertion by Puig.

    \paragraph{Pro-fusion systems.}
    Stancu and Symonds \cite{StancuSymonds} define a \emph{pro-fusion system} on a pro-$p$ group $P$ starting from an inverse system of fusion systems on finite $p$-groups. Just like in the finite case what they define is a category whose objects are the closed subgroups of $P$ and whose morphism sets consist of injective  continuous group homomorphisms. However, in contrast to the finite case, there is no list of axioms characterizing when such a category is a pro-fusion system -- to prove that it is, one has to realize it via an inverse system of fusion systems on finite $p$-groups. 

    \begin{defi}\label{defi profusion system}
        Let $P$ be a pro-$p$ group, and assume $P=\varprojlim_i P_i$ where $i$ ranges over some directed set $I$ and each $P_i$ is finite. Assume we are given an inverse system of fusion systems $\mathcal F_i$ on $P_i$ for each $i$. Then this defines a category $\mathcal F=\varprojlim_i \mathcal F_i$ where
        \begin{enumerate}
            \item the objects are the closed subgroups of $P$, and
            \item for any two closed subgroups $R, S\leq P$ we set
                \begin{equation}
                    \Hom_{\mathcal F}(R, S) = \varprojlim_i \Hom_{\mathcal F_i}(R_i, S_i),
                \end{equation}
                where $R_i$ and $S_i$ denote the respective images of $R$ and $S$ in $P_i$, and all maps are the ones induced by the inverse system.
        \end{enumerate}
        A category $\mathcal F$ obtained by this construction is called a \emph{pro-fusion system}.  
    \end{defi}

    We can define morphisms of pro-fusion systems just like morphisms of fusion systems.

    \begin{defi}
        Let $\mathcal F$ and $\mathcal F'$ be pro-fusion systems on pro-$p$ groups $P$ and $Q$, respectively. A \emph{morphism} from $\mathcal F$ to $\mathcal F'$ is a pair $(\alpha,\Phi)$, where $\alpha:\ P\longrightarrow Q$ is a continuous group homomorphism and  $\Phi:\ \mathcal F \longrightarrow \mathcal F'$ is a functor such that
        \begin{enumerate}
            \item $\alpha(R)=\Phi(R)$ for all closed subgroups $R\leq P$, and
            \item $\Phi(\varphi)\circ \alpha=\alpha\circ \varphi$ for all $\varphi\in\Hom_{\mathcal F}(R,S)$, where $R,S\leq P$ are closed.
        \end{enumerate}
    \end{defi}
    
    Note that the $\Hom$-sets in $\mathcal F$ and $\mathcal F'$ are topological spaces, so it would be reasonable to ask that $\Phi$ induce continuous maps between $\Hom$-spaces, rather than just maps of sets. This is unnecessary though, since $\alpha$ determines $\Phi$ just like in the finite case, and continuity is automatic. 
    Note that the above definition turns the collection of all pro-fusion systems into a category, and
    Stancu and Symonds show \cite[Section 3]{StancuSymonds} that $\varprojlim_i \mathcal F_i$ as defined above is indeed an inverse limit in this category, justifying the notation. 

    \begin{defi}[{see \cite[Definition 4.1]{StancuSymonds}}]
        A pro-fusion system is called \emph{pro-saturated} if it is isomorphic to an inverse limit of saturated fusion systems on finite $p$-groups.
    \end{defi}
    All pro-fusion systems we construct in the present paper are inverse limits of saturated fusion systems (since block fusion systems are known to be saturated), so they will all automatically be pro-saturated. Pro-saturation has the following interesting consequence, which is useful when trying to describe pro-fusion systems explicitly rather than as an inverse limit.
    \begin{prop}[{see \cite[Section 4.6]{StancuSymonds}}]
        Let $\mathcal F$ and $\mathcal F'$ be pro-saturated pro-fusion systems on a pro-$p$ group $P$. If the full subcategories of $\mathcal F$ and $\mathcal F'$ whose objects are the \emph{open} subgroups of $P$ coincide, then $\mathcal F=\mathcal F'$.  
    \end{prop}

\section{Block pro-fusion systems for countably based profinite groups}

In this section we will attach a pro-saturated pro-fusion system to a block of a countably based profinite group. By Corollary~\ref{corol can take G countably based} this effectively encompasses all blocks of arbitrary profinite groups with countably based defect groups. The definition will resemble the definition in the finite case, but the relationship between the pro-fusion system of a block of a profinite group and the fusion systems of the corresponding blocks of the finite quotients is not as straightforward as one might expect. In particular, one cannot define
this pro-fusion system as the inverse limit of the fusion systems of the corresponding blocks of finite quotients. The latter simply do not fit into an inverse system.

On the level of finite groups, the main ingredients needed are Lemmas~\ref{lemma nu minus}~and~\ref{lemma fusion system grows} below.
Lemma~\ref{lemma nu minus} summarizes what happens to block idempotents and Brauer pairs under taking quotients. This then feeds into Lemma~\ref{lemma fusion system grows}, which describes how fusion systems of the finite-dimensional quotients of the block fit together. This will allow us to construct inverse systems of fusion systems. The crucial assumption in both Lemma~\ref{lemma nu minus}~and~\ref{lemma fusion system grows} is that we only consider quotients $G/N$ of $G$ such that $N\cap Q$ is a Sylow $p$-subgroup of $N$, where $Q$ is some $p$-subgroup of $G$ (e.g. the defect group of a block). While this assumption looks rather arbitrary for finite groups, when looking at quotients of profinite groups it will translate to asking that $N$ be sufficiently small, which is a natural assumption in the profinite context.

    \begin{lemma}\label{lemma nu minus}
        Let $G$ be a finite group and let $N\unlhd G$ be a normal subgroup. Let $
            \nu:\ kG \longrightarrow kG/N
        $
        denote the natural epimorphism. For a $p$-subgroup $Q$ of $G$ such that $Q\cap N$ is a Sylow $p$-subgroup of $N$, define
        \begin{equation}
        C_{Q,N}=\left\{ g \in G \ :\ [g,Q] \subseteq Q \cap N  \right\}.
        \end{equation}
        For each such $Q$ there is a map
        \begin{equation}
            \nu^-_Q: \{\textrm{ prim. idempot. of $Z(kC_{G/N}(QN/N))$ }\} \longrightarrow 
            \{\textrm{ prim. idempot. of $Z(kC_{G}(Q))^{C_{Q, N}}$ }\}
        \end{equation}
         such that the following hold:
        \begin{enumerate}
            \item If $e$ is a primitive idempotent in $Z(kC_{G/N}(QN/N))$, then  \begin{equation}
                \nu(\nu^-_Q(e))\cdot e=e,\label{eqn defining nuQ}
            \end{equation}
            and this property uniquely determines $\nu^-_Q(e)$.
            Moreover, $\nu^-_Q$ is $G$-equivariant, that is, 
            \begin{equation}
                \nu^-_Q(e)^g=\nu^-_{Q^g}(e^g) \quad \textrm{ for all $g\in G$. }
            \end{equation}
            \item If $P\leq Q$ are two $p$-subgroups of $G$ such that $P\cap N$ is a Sylow $p$-subgroup of $N$, and $(PN/N,c)\leq (QN/N,d)$ are two Brauer pairs for $kG/N$, then 
            for any two Brauer pairs $(P,\tilde c)$ and $(Q,\tilde d)$ such that $\nu^-_P(c)\cdot \tilde c \neq 0$ and $\nu^-_Q(d)\cdot \tilde d \neq 0$ there is an $x\in C_{P,N}$ such that 
            \begin{equation}
                (P, \tilde c^x) \leq (Q, \tilde d)
            \end{equation} 
            as Brauer pairs for $kG$.
        \end{enumerate}
    \end{lemma}
    \begin{proof}
        Before we start we should point out that $C_G(Q)\leq C_{Q,N}\leq N_G(Q)$, so $Q$ is normalized by $C_{Q,N}$. It will become clear below that our assumptions imply $C_{Q,N}N/N=C_{G/N}(QN/N)$.

        Now let $P\unlhd Q$ be two $p$-subgroups of $G$ such that $P\cap N$ is a Sylow $p$-subgroup of $N$. Note that $[C_{Q,N}, P]\subseteq Q \cap N=P\cap N$, so $C_{Q,N}\leq C_{P,N}$.
        Note also that for $q\in Q$ we have 
        \begin{equation}
        {C_{P,N}}^q=\left\{ g \in G \ :\ [g^{q^{-1}},P] \subseteq P \cap N\right\} = \left\{ g \in G \ :\ [g,P^q] \subseteq P^q \cap N\right\}= C_{P,N}. 
        \end{equation}
        That is, $Q$ normalises $C_{P,N}$, which implies that $QC_{P,N}=C_{P,N}Q$ is a group.
        We claim that there is a commutative diagram
        \begin{equation}\label{eqn comm diagram brauer maps}
            \xymatrix{
                Z(kC_G(P))^{C_{P,N}Q} \ar[rr]^{\Br_Q} \ar@{->}[d]_\nu && Z(kC_G(Q))^{C_{Q,N}} \ar@{->}[d]^\nu \\
                Z(kC_{G/N}(PN/N))^Q \ar[rr]^{\Br_{QN/N}}&& Z(kC_{G/N}(QN/N)) 
            }
        \end{equation}
        where $\Br_Q$ and $\Br_{QN/N}$ denote the respective Brauer maps. The $\nu$'s in this diagram are just the restriction of the natural epimorphism $\nu$ from the statement of the lemma. However, these restrictions are typically not surjective, and well-definedness is not clear. 
        
        To show that the vertical maps are well-defined, we will just show that  $\nu(Z(kC_G(P))^{C_{P,N}}) \subseteq  Z(kC_{G/N}(PN/N)) $, to avoid having to show the same thing for both vertical arrows. Since $\nu$ is clearly $Q$-equivariant, it will then send $Q$-invariants to $Q$-invariants.
        Note that $C_G(P)\unlhd C_{P,N}$, and therefore $Z(kC_G(P))^{C_{P,N}}$ is spanned by elements of the form 
        \begin{equation}
            \widehat g = \sum_{x\in g^{C_{P,N}}} x \quad \textrm{for $g\in C_G(P)$.}
        \end{equation} 
        Take such a $\widehat g$ as well as an $hN\in C_{G/N}(PN/N)$. Our aim is to show that $\nu(\widehat g)$ commutes with $hN$. We have $P^hN=PN$, since $hN$ centralises $PN/N$. Now $P$ is a Sylow $p$-subgroup of $PN$, since $|PN|=|PN/N||N|=|P/P\cap N||N|$ and $|P|=|P/P\cap N||P\cap N|$, implying that the index of $P$ in $PN$ is equal to the index of $P\cap N$ in $N$, which is coprime to $p$ by assumption ($P\cap N$ is a Sylow $p$-subgroup of $N$). Since it has the same order as $P$, the group $P^h$ is also a Sylow $p$-subgroup of $PN$. So $P$ and $P^h$ are conjugate within $PN$. This means there is an $x\in P$ and an $n\in N$ such that $P^{xn^{-1}}=P^h$, which implies $P^{hn}=P$. It follows that $[hn, P]\subseteq P$. Since $hN\in C_{G/N}(PN/N)$ it follows that $[hn, P]\subseteq N$, and therefore $hn\in C_{P,N}$. So $\widehat g^{hn}=\widehat g$, which implies $\nu (\widehat g)^{hN}=\nu(\widehat g)$. Since  $hN\in C_{G/N}(PN/N)$ was arbitrary it follows that $\nu (\widehat g)\in   Z(kC_{G/N}(PN/N))$. The same holds with $P$ replaced by $Q$.
        
        The top horizontal arrow is also well-defined, since $\Br_Q$ as defined in \eqref{eqn def Brauer map} is clearly $N_G(Q)\cap N_G(P)$-equivariant, and $C_{Q,N}\leq N_G(Q)\cap N_G(P)$. Since $C_{Q,N}\leq C_{P,N}$ we have $Z(kC_G(P))^{C_{P,N}Q}\subseteq Z(kC_G(P))^{QC_{Q,N}}$.
        
        One can check commutativity directly by verifying commutativity of 
        \begin{equation}
            \xymatrix{
                kC_G(P)^Q \ar[rr]^{\Br_Q} \ar@{->}[d]_\nu && kC_G(Q) \ar@{->}[d]^\nu \\
                kC_{G/N}(PN/N)^Q \ar[rr]^{\Br_{QN/N}}&& kC_{G/N}(QN/N), 
            }
        \end{equation}
        which is classical. Concretely, assume $gN\in C_{G/N}(QN/N)$, $g\in C_G(P)$ but $g\not\in C_G(Q)$.
         Pick $q\in Q$ with $g^q\neq g$. Then $\nu$ sends the $Q$-orbit sum of $g$ to the $Q$-orbit sum of $\sum_{x\in g^{\langle q \rangle}} x$, but the latter maps to $|g^{\langle q \rangle}|\cdot gN=0$.  So $\nu$ maps the $Q$-orbit sum of $g$ to  zero. 

        Now let us define $\nu^-_Q$.
        If $1=e_1+\ldots+e_n$ is a decomposition of $1$ as a sum of primitive orthogonal idempotents in $Z(kC_G(Q))^{C_{Q,N}}$, then $1=\nu(e_1)+\ldots +\nu(e_n)$ is a decomposition of $1$ as a sum of orthogonal (but not necessarily primitive) idempotents in $Z(kC_{G/N}(QN/N))$. Therefore, given any primitive idempotent $c\in Z(kC_{G/N}(QN/N))$ there exists a unique $i$ such that $\nu(e_i)\cdot c \neq 0$, and then necessarily $\nu(e_i)\cdot c = c$. We define 
        \begin{equation}
            \nu_Q^{-}(c) := e_i.
        \end{equation}
        This defines map from primitive idempotents of $Z(kC_{G/N}(QN/N))$ to primitive idempotents of  $Z(kC_G(Q))^{C_{Q,N}}$, and the equality \eqref{eqn defining nuQ} is clearly satisfied and uniquely determines $\nu^-_Q$. It is also clear that $\nu_Q^-$ is $G$-equivariant. 
        
        It remains to prove the second part of our assertion. Let us first assume, as we did above, that $P$ is normal in~$Q$.
        By \cite[Proposition 6.3.4]{LinckelmannVolII}~or~\cite[Theorem 3.4]{AlperinBroue} $(PN/N,c)\leq(QN/N,d)$ is equivalent to $c$ being the unique $Q$-stable primitive idempotent in $Z(kC_G(PN/N))$ with $\Br_{QN/N}(c)\cdot d \neq 0$. 
        Now, due to $G$-equivariance of $\nu^-_P$, the idempotent $\nu^-_P(c)\in Z(kC_G(P))^{C_{P,N}}$ is automatically $Q$-invariant, that is,  $\nu^-_P(c)\in Z(kC_G(P))^{C_{P,N}Q}$. Moreover, using the commutative diagram \eqref{eqn comm diagram brauer maps}
        \begin{equation}
            \nu(\Br_Q(\nu^-_P(c)) \cdot \nu^-_Q(d))= \nu (Br_Q(\nu^-_P(c))) \cdot \nu(\nu^-_Q(d))= \Br_{QN/N}(\nu(\nu^-_P(c)))\cdot \nu(\nu^-_Q(d)).
        \end{equation} 
        Now, using $ c\cdot \nu(\nu^-_P(c)) =c$ and the analogous fact for $d$,
        \begin{equation}
            \begin{array}{rcl}
            \Br_{QN/N}(\nu(\nu^-_P(c)))\cdot \nu(\nu^-_Q(d)) \cdot \Br_{QN/N}(c)\cdot d &=& 
            \Br_{QN/N}(c \cdot \nu(\nu^-_P(c)))\cdot d\cdot \nu(\nu^-_Q(d)) \\
            &=& \Br_{QN/N}(c)\cdot d \\ &\neq& 0,
            \end{array}
        \end{equation}
        which implies that $\nu(\Br_Q(\nu^-_P(c)) \cdot \nu^-_Q(d))\neq 0$, since the above was obtained by multiplying this by another expression, namely $ \Br_{QN/N}(c)\cdot d$. 
        It follows that $\Br_Q(\nu^-_P(c)) \cdot \nu^-_Q(d)\neq 0$.

        Now $\nu^-_Q(d)$ is necessarily the sum over the $C_{Q,N}$-orbit of $\tilde d$. We also know by \cite[Proposition 6.3.4]{LinckelmannVolII} that there exists a $\tilde c'\in Z(kC_G(P))^Q$, primitive in $Z(kC_G(P))$, such that $\Br_Q(\tilde c') \cdot \tilde d \neq 0$.
        Since $\tilde c'$ is primitive and $Q$-invariant, its $C_{P,N}Q=QC_{P,N}$-orbit is actually the same as its $C_{P,N}$-orbit (the subtlety here is that $C_{P,N}Q$ does not act on $Z(kC_G(P))^Q$, but it does act on $Z(kC_G(P))$, and $\tilde c'$ happens to be a primitive idempotent in both). So let us denote the sum over the $C_{P,N}$-orbit of $\tilde c'$ by $\tilde c''$. Then $\tilde c''$ is a primitive idempotent in $Z(kC_G(P))^{C_{P,N}Q}$ and $\Br_Q(\tilde c'')\cdot \nu^-_Q(d) \neq 0$. By uniqueness, we must have $\tilde c''=\nu^-_P(c)$, which means that $\nu^-_P(c)$ is a sum of all $C_{P,N}$-conjugates of $\tilde c'$. In particular, $\tilde c$ is a $C_{P,N}$-conjugate of $\tilde c'$, which proves the assertion, although still under the assumption $P\unlhd Q$.
        
        If $P$ is not normal in $Q$, then there exists a chain $P=P_1\unlhd P_2\unlhd \ldots \unlhd P_r = Q$. By \cite[Theorem 6.3.3]{LinckelmannVolII} there exist unique primitive idempotents $c_i \in Z(kC_{G/N}(P_iN/N))$ such that $(P_iN/N,c_i)\leq (P_{i+1}N/N,c_{i+1})$ for all $1\leq i \leq r-1$ and $(P_rN/N,c_r)=(QN/N,d)$. By uniqueness it follows that $(P_1N/N,c_1)=(PN/N,c)$. For each of the $P_i$ for $2\leq i \leq r-1$ we can pick a primitive idempotent $\tilde c_i\in Z(kC_G(P_i))$ such that $\nu^-_{P_i}(c_i)\cdot \tilde c_i\neq 0$. Set $\tilde c_1=\tilde c$ and $\tilde c_r=\tilde d$. Then we already saw that for each $1\leq i < r$ there is an $x_i\in C_{P_i,N}\leq C_{P,N}$ such that $(P_i,\tilde c_i^{x_i})\leq (P_{i+1},\tilde c_{i+1})$. But then our claim holds with $x=x_1\cdots x_{r-1}$.
    \end{proof}

    The next lemma shows that, under suitable hypotheses, the fusion systems of blocks of $G/N$ and those of blocks of $G$ fit together nicely as long as we choose Brauer pairs compatibly.
    
    \begin{lemma}\label{lemma fusion system grows}
        Let $G$ be a finite group, $N \unlhd G$ a normal subgroup and $D\leq G$ a $p$-subgroup. Assume that $N\cap D$ is a Sylow $p$-subgroup of $N$, and that $b\in Z(kG/N)$ is a block idempotent such that $kG/Nb$ has defect group $DN/N$ and $kG\tilde b$  has defect group $D$, where $\tilde b\in Z(kG)$ is the unique block idempotent such that $\nu(\tilde b)\cdot b =b$. Let $(DN/N,e)$ be a maximal $(G/N, b)$-Brauer pair.
        \begin{enumerate}
        \item  If $\tilde e$ is a primitive idempotent in $Z(kC_G(D))$ such that $\nu^-_D(e)\cdot \tilde e \neq 0$, then $(D,\tilde e)$ is a maximal $(G,\tilde b)$-Brauer pair and
        \begin{equation}
            \mathcal F_{(DN/N,e)}(G/N,b) \leq \mathcal F_{(D,\tilde e)}(G, \tilde b)/D\cap N,
        \end{equation}
        where both are viewed as fusion systems on $DN/N$. 
        \item If $\tilde e'$ is another primitive idempotent in $Z(kC_G(D))$ such that $\nu^-_D(e)\cdot \tilde e' \neq 0$, then 
        \begin{equation}
            \mathcal F_{(D,\tilde e)}(G, \tilde b)/D\cap N = \mathcal F_{(D,\tilde e')}(G, \tilde b)/D\cap N.
        \end{equation}
        \end{enumerate}
    \end{lemma}
    \begin{proof}
        Let us make a few remarks before we start. The map $\nu: Z(kG) \longrightarrow Z(kG/N)$ is well-defined and $D$ acts trivially on domain and range. Therefore $\nu\circ \Br_D=\Br_{DN/N}\circ \nu$ (this is diagram~\eqref{eqn comm diagram brauer maps} with $P=\{1\}$ and $Q=D$). Note that the assumption  $\nu^-_D(e)\cdot \tilde e \neq 0$ implies  $\nu^-_D(e)\cdot \tilde e =\tilde e$, and $\tilde e$ is a primitive idempotent in $Z(kC_G(D))$ whilst $\nu^-_D(e)$ is a primitive idempotent in $Z(kC_G(D))^{C_{D,N}}\subseteq Z(kC_G(D))$. So $\nu^-_D(e)$ is the sum of all $C_{D,N}$-conjugates of $\tilde e$. In particular, if $\Br_D(\tilde b)\cdot \tilde e=0$, then $\Br_D(\tilde b)\cdot \nu_D^-(e)=0$, as $\tilde b$ is invariant under conjugation by $C_{D,N}$ (or by $G$, for that matter).
        
        Now let us show that $(1N/N, b)\leq (DN/N, e)$ implies  $(1,\tilde b)\leq (D,\tilde e)$. Assume by way of contradiction that $\Br_D(\tilde b)\cdot \tilde e=0$. By the above it follows that $\Br_D(\tilde b)\cdot \nu_D^-(e)=0$ and therefore 
        $$\nu(\Br_D(\tilde b)\cdot \nu_D^-(e))=\Br_{DN/N}(\nu(\tilde b))\cdot \nu(\nu_D^-(e))=0.$$
        The assumption $(1N/N, b)\leq (DN/N, e)$ means that $\Br_{DN/N}(b)\cdot e \neq 0$. Multiplying the above by ${\Br_{DN/N}(b)\cdot e}$ implies $\Br_{DN/N}(b)\cdot e=0$, since $\nu(\tilde b)\cdot b=b$ and $\nu(\nu_D^-(e))\cdot e=e$. This is a contradiction. Since $D$ is a defect group of $kG\tilde b$ by assumption, it is now also clear that $(D,\tilde e)$ is a maximal Brauer pair. 

        For each  $D\cap N\leq  P\leq D$ take the (unique) primitive idempotent $e_{PN/N}\in Z(kC_{G/N}(PN/N))$ such that $(PN/N, e_{PN/N})\leq (DN/N, e)$. By Lemma~\ref{lemma nu minus}, the unique primitive idempotent $\tilde e_P \in Z(kC_G(P))$ such that $(P,\tilde e_P)\leq (D,\tilde e)$ satisfies $\nu^-_P(e_{PN/N})\cdot \tilde e_P\neq 0$. 
        
        Now consider two subgroups $D\cap N \leq P\leq Q\leq D$. The elements of $\Hom_{\mathcal F_{(DN/N,e)}(G/N,b)}(PN/N,QN/N)$ are, by definition, group homomorphisms induced by conjugation by elements $g\in G$ such that $P^gN/N\leq QN/N$ and $e_{PN/N}^g= e_{P^gN/N}$. Since $Q\cap N=D\cap N$ is a Sylow $p$-subgroup of $N$, we have $|QN|=|Q/N\cap Q||N|$ and $|Q|=|Q/N\cap Q||N\cap Q|$ and therefore $Q$ is a Sylow $p$-subgroup of $QN$. Since $P^g$ is a $p$-subgroup of $QN$, we can find $n\in N$ and $x\in Q$ such that $P^{g}\leq Q^{xn^{-1}}$, which implies $P^{gn}\leq Q$. In particular $\nu^-_{P^{gn}}(e_{P^{gn}N/N})\cdot \tilde e_{P}^{gn}\neq 0$ by $G$-equivariance of $\nu^-_P$. Now Lemma~\ref{lemma nu minus} guarantees the existence of an $x\in C_{P^{gn},N}$ such that $(P^{gn},\tilde e_{P}^{gnx})\leq (Q, \tilde e_Q)$, or, equivalently, $\tilde e_{P}^{gnx}=\tilde e_{P^{gnx}}$. Note that conjugation by $g$ and conjugation by $gnx$ induce the same group homomorphism from $PN/N$ to $QN/N$.
        
        The elements of 
        $\Hom_{ \mathcal F_{(D,\tilde e)}(G, \tilde b)/D\cap N}(PN/N, QN/N)$ are, by definition, group homomorphisms induced by conjugation by elements $h\in G$ such that $P^h\leq Q$ and $\tilde e_{P}^h=\tilde e_{P^h}$. In particular, the element $gnx$ from the previous paragraph induces an element of   $\Hom_{ \mathcal F_{(D,\tilde e)}(G, \tilde b)/D\cap N}(PN/N, QN/N)$. This proves that $\mathcal F_{(DN/N,e)}(G/N,b) \leq \mathcal F_{(D,\tilde e)}(G, \tilde b)/D\cap N$.

        If $\tilde e'$ is another primitive idempotent in $Z(kC_G(D))$ such that $\nu^-_D(e)\cdot \tilde e' \neq 0$, then $\tilde e' = \tilde e^x$ for some $x\in C_{D,N}$. It follows that $h$ induces an element of $\Hom_{ \mathcal F_{(D,\tilde e)}(G, \tilde b)/D\cap N}(PN/N, QN/N)$  if and only if $x^{-1}hx$ induces an element of $\Hom_{ \mathcal F_{(D,\tilde e')}(G, \tilde b)/D\cap N}(PN/N, QN/N)$. But $h$ and $x^{-1}hx$ induce the same group homomorphism from $PN/N$ to $QN/N$. So the second part of our claim follows.
    \end{proof}

    We are now ready to look at blocks of profinite groups. We restrict ourselves to countably based profinite groups, that is, groups $G$ such that $G=\varprojlim_{i\in\N} G/N_i$ for a chain of open normal subgroups $N_i$ intersecting in $1$. For technical reasons we choose a particular chain below, but we will see in Proposition~\ref{prop independence of inverse system} that the block pro-fusion systems we construct are independent of this choice.

    \begin{notation}\label{notation setup}
        Let $G$ be a countably based profinite group, and fix a chain 
        $$N_1\geq N_2 \geq N_3\geq \ldots$$
        of open normal subgroups of $G$ such that $\bigcap_{i\in \N} N_i=1$. Let $b\in Z(k\dbl G\dbr )$ be a block idempotent and let $D$ be a defect group for $k\dbl G\dbr b$.
    \end{notation}
    We will keep this setup and notation for the rest of this section. We are now ready to define Brauer pairs for profinite groups, although the definition does not match what one would naively expect. We address this briefly in Section~\ref{section Brauer pair naive}.

    \begin{defi}\label{defi brauer pair profinite}
        A \emph{Brauer pair} for $k\dbl G\dbr $ is a pair $(P, \hat e)$, where 
        \begin{enumerate}
        \item $P$ is an open subgroup of a Sylow $p$-subgroup of $G$,
        \item $\hat e = [(e_i)_{i\in \N}]_\sim$ is an equivalence class of sequences of primitive idempotents $e_i \in Z(kC_{G/N_i}(PN_i/N_i))$, where we say $\hat e \sim \hat e'$ if  $e_i=e_i'$ for all but finitely many $i$. 
        
        We require that, for all but finitely many $i$ such that $N_i\cap P$ is a Sylow $p$-subgroup of $N_i$, we have
        \begin{equation}\label{eqn brauer pair idempots}
        \nu^{-}_{PN_{i+1}/N_{i+1}}(e_i)\cdot e_{i+1}\neq 0.
        \end{equation}
        
        \end{enumerate}
        We say $(P, \hat e)\leq (Q,\hat f)$ if  $(PN_i/N_i, e_i)\leq (QN_i/N_i,f_i)$ for all but finitely many $i$.
    \end{defi}

    Note that $P$ being open in a Sylow $p$-subgroup of $G$ implies that $N_i\cap P$ is a Sylow $p$-subgroup of $N_i$ for all~$i$ sufficiently large. The condition is required in order that the map $\nu^{-}_{PN_{i+1}/N_{i+1}}$ appearing in equation~\eqref{eqn brauer pair idempots} is defined.


 \begin{remark}\label{remark blocks}
    By \cite[Corollary 5.10]{FranquizMacQuarrieBrauer} there is an open normal subgroup $N_0\unlhd G$ together with a block idempotent ${b_0\in Z(kG/N_0)}$ with the following properties:
    \begin{enumerate}
        \item $\nu(b)\cdot b_0\neq 0$, where ${\nu:\ Z(k\db{G})\longrightarrow Z(kG/N_0)}$ is the natural map.
        
        \item For any open normal subgroup $N\unlhd G$ contained in $N_0$, the block $kG/Nb_N$ has defect group $DN/N$, where if ${\nu:\ Z(kG/N)\longrightarrow Z(kG/N_0)}$ is the natural map, then $b_N$ denotes the unique block idempotent in $Z(kG/N)$ such that $\nu(b_N)\cdot b_0\neq 0$.
    \end{enumerate}
    Furthermore,
    $$
        k\dbl G\dbr b=\varprojlim_{N\subseteq N_0} kG/Nb_N.
    $$
    \end{remark}
    

    \begin{defi}\label{defi brauer pairs blocks}
        Notation as above. For each $i>0$ such that $N_i\subseteq N_0$ define $b_i=b_{N_i}$, and choose $b_i$ arbitrarily for all (finitely many) other $i$. We define
        $$
            \hat b = [(b_i)_{i\in \N}]_\sim.
        $$
        We call a Brauer pair $(P,\hat e)$ a \emph{$(G,b)$-Brauer pair} if $(1,\hat b)\leq (D, \hat e)$.
    \end{defi}

    The above definition implicitly uses that $\hat b$ satisfies equation~\eqref{eqn brauer pair idempots} in Definition~\ref{defi brauer pair profinite}. This is however immediate from the definitions. 

    

    \begin{prop}\label{prop profinite brauer pairs}
        \begin{enumerate}
        \item The set of all Brauer pairs for $k\dbl G\dbr $  form a poset with a $G$-action.
        \item Let $P\leq Q\leq G$ be open in a Sylow $p$-subgroup of $G$, and let $(Q,\hat d)$ be a Brauer pair for $k\dbl G\dbr $. Then there exists a \emph{unique} Brauer pair $(P,\hat c)$ such that $(P,\hat c)\leq (Q,\hat d)$.  
        \item\label{prop profinite brauer pairs:maximal} There is a maximal $(G,b)$-Brauer pair $(D,\hat e)$ such that $(DN_i/N_i,e_i)$ is a maximal $(G/N_i, b_i)$-Brauer pair for all but finitely many $i$. If $(D,\hat e')$ is another maximal $(G,b)$-Brauer pair, then $(D, \hat e')=(D, \hat e)^g$ for some $g\in G$. 
        \end{enumerate}
    \end{prop}
    \begin{proof}
        All the axioms of a partial order are immediate for ``$\leq$'', and there is a natural $G$-action. For the second point we let $c_i$ be the unique primitive idempotent in $Z(kC_{G/N_i}(PN_i/N_i))$ such that $(PN_i/N_i, c_i)\leq (QN_i/N_i, d_i)$, for each $i\in \N$. It is clear that $\hat c$ must have this form for all but finitely many $i$, so uniqueness is immediate. We just need to show that $(P,\hat c)$ is in fact a Brauer pair. If we choose $i_0$ such that $N_{i_0}\cap P$ is a Sylow $p$-subgroup of $N_{i_0}$, then Lemma~\ref{lemma nu minus} guarantees $\nu^{-}_{PN_{i+1}/N_{i+1}}(c_i)\cdot c_{i+1}\neq 0$ for all $i\geq i_0$, so the condition in Definition~\ref{defi brauer pair profinite} is satisfied.

        For the third point we first pick $i_0\in \N$ sufficiently large such that, for all $i\geq i_0$, the group $D\cap N_i$ is a Sylow $p$-subgroup of $N_i$ and $kG/N_ib_i$ has defect group $DN_i/N_i$ (see Remark~\ref{remark blocks}).
        Now we choose a maximal $(G/N_{i_0},b_{i_0})$-Brauer pair $(DN_{i_0}/N_{i_0}, e_{i_0})$ and we choose $e_i$ for $i > i_0$ inductively such that 
        $$\nu^-_{DN_{i+1}/N_{i+1}}(e_{i})\cdot e_{i+1}\neq 0.$$ Lemma~\ref{lemma fusion system grows} ensures that each $(DN_i/N_i, e_i)$ is a $(G,b_i)$-Brauer pair, and it is maximal since $DN_i/N_i$ is a defect group of $kG/N_ib_i$. We then define $\hat e=[(e_i)_{\in \N}]_\sim$, where we pick $e_i$ for $i<i_0$ arbitrarily. This shows the existence of $(D, \hat  e)$. It is also clear that this $(D, \hat  e)$ is maximal.
        
        If there is another such Brauer pair $(D,\hat e')$, then there must be some $i_1\geq i_0$ such that $(DN_{i}/N_{i}, e'_{i})$ is a $(G/N_{i}, b_{i})$-Brauer pair for all $i\geq i_1$, and these Brauer pairs will automatically be maximal since $DN_{i}/N_{i}$ is a defect group. So there is
        an $x_{i_1}\in G$ such that $e_{i_1}^{x_{i_1}}=e_{i_1}'$ , since all maximal $(G/N_{i_1},b_{i_1})$-Brauer pairs are conjugate. Assume by way of induction that for some $i\geq i_1$ we have $x_{i_1},\ldots,x_i\in G$ such that $e_j^{x_{i_1}\cdots x_i}=e'_{j}$ for all $j\leq i$. By Lemma~\ref{lemma nu minus} there is an $\bar x_{i+1}\in C_{DN_{i+1}/N_{i+1}, N_i/N_{i+1}}\leq G/N_{i+1}$ such that $e_{i+1}^{x_1\cdots x_{i+1}}=e'_{i+1}$, where $x_{i+1}$ denotes a preimage of $\bar x_{i+1}$ in $G$. Note that $C_{DN_{i+1}/N_{i+1}, N_i/N_{i+1}}N_i/N_{i}= C_{G/N_i}(DN_i/N_i)$, and therefore $e_{j}^{x_1\cdots x_{i+1}}=e'_{j}$ for all $j\leq i$ (since $e_j\in Z(kC_{G/N_j}(DN_j/N_j))$).
        Now note that the set $T_i = \{ g \in G \ :\ e_i^g=e_i'\}$ is closed for any $i$, since $T_i=T_iN_i$. We just showed that every finite intersection of $T_i$'s is non-empty. By compactness of $G$ that means the intersection of all $T_i$ is non-empty. An element $x$ in this intersection satisfies $\hat e^x=\hat e'$ by definition.
    \end{proof}

    We can now define the pro-fusion system of a block of a profinite group, analogous to the finite group case. However, it will not immediately be clear that this is an inverse limit of fusion systems on finite $p$-groups, and the objects of the category we define are only the open subgroups of a defect group rather than the closed ones. 

    \begin{defi}\label{def via brauer pairs}
        Let $(D,\hat e)$ be a maximal $(G,b)$-Brauer pair. 
        For any open subgroup $P\leq D$ let $(P, \hat e_P)$ denote the unique Brauer pair such that $(P,\hat e_P)\leq (D,\hat e)$. 
        Define a category $\mathcal F=\mathcal F_{(D,\hat e)}(G,b)$ as follows:
        \begin{enumerate}
            \item The objects of $\mathcal F$ are the open subgroups of $D$.
            \item If $P$ and $Q$ are open subgroups of $D$ we define
                $\Hom_{\mathcal F}(P,Q)$ to be the set of group homomorphisms from $P$ to $Q$ induced by conjugation by elements $g\in G$ such that $(P, \hat e_{P})^g\leq (Q,\hat e_{Q})$.
        \end{enumerate}
    \end{defi}

    Note that pro-saturated pro-fusion systems in the sense of \cite{StancuSymonds} are fully determined by their restriction to open subgroups (see the remarks in \cite[Section 4.6]{StancuSymonds}), so we do not need to give an explicit description for homomorphisms between closed subgroups. The inverse limit in Theorem~\ref{thm fusion as inverse limit} below is obviously also defined on closed subgroups, so one could extend the definition, but it would not look as clean as the one above.

    \begin{thm}\label{thm fusion as inverse limit}
        Let $(D,\hat e)$ be a maximal $(G,b)$-Brauer pair, and define $\mathcal F_i = F_{(DN_i/N_i,e_i)} (G/N_i, b_i)$.  Then there is a strictly increasing function $\mu:\ \N \longrightarrow \N$ and an $i_0\in \N$ such that 
        \begin{equation}\label{eqn block fusion as inverse limit}
            \mathcal F_{(D,\hat e)}(G, b)=\varprojlim_{i\geq i_0} \mathcal F_{\mu(i)}/((D\cap N_i)N_{\mu(i)}/N_{\mu(i)}),
        \end{equation}
        and $\mathcal F_{(D,\hat e)}(G, b)/D\cap N_i=\mathcal F_{\mu(i)}/((D\cap N_i)N_{\mu(i)}/N_{\mu(i)})$ as fusion systems on~$D/D\cap N_i$ for all $i\geq i_0$.
    \end{thm}
    \begin{proof}
        Write $\mathcal F=\mathcal F_{(D,\hat e)}(G, b)$.
        Note that $(D\cap N_i)N_{\mu(i)}=DN_{\mu(i)}\cap N_i$ (this can be seen elementarily). So $\mathcal F_{\mu(i)}/((D\cap N_i)N_{\mu(i)}/N_{\mu(i)})$ is a fusion system on $DN_{\mu(i)}/(DN_{\mu(i)}\cap N_i)\cong DN_i/N_i\cong D/D\cap N_i$, and we view it as a fusion system on $DN_i/N_i$. 

        We start by constructing the inverse system. Pick $i_0$ so that for all $i\geq i_0$ we have that ${\nu^-_{DN_{i+1}/N_{i+1}}(b_{i})\cdot b_{i+1}\neq 0}$, the block $kG/N_ib_i$ has defect group $DN_i/N_i$ and $D\cap N_i$ is a Sylow $p$-subgroup of~$N_i$.
        Lemma~\ref{lemma fusion system grows} implies that
        $$\mathcal F_j/((D\cap N_i)N_{j}/N_{j})\leq \mathcal F_{j+1}/((D\cap N_i)N_{j+1}/N_{j+1})$$ 
        for all $j\geq i\geq i_0$, giving us an ascending chain of fusion systems on the finite group $DN_i/N_i$. This chain must eventually become stationary, implying that if we choose $\mu(i)$ for $i\geq i_0$ large enough then
        \begin{equation}\label{eqn upward equality fusion systems}
        \Hom_{\mathcal F_{\mu(i)}/((D\cap N_i)N_{\mu(i)}/N_{\mu(i)})} (PN_i/N_i,QN_i/N_i)=\Hom_{\mathcal F_j/((D\cap N_i)N_{j}/N_{j})} (PN_i/N_i,QN_i/N_i)
        \end{equation}
        for all $j\geq \mu(i)$ and all $D\cap N_i\leq P, Q\leq D$. Of course we can simultaneously ensure that $\mu$ is strictly increasing.

        By Proposition~\ref{prop factor out strongly closed} the natural maps induce a morphism of fusion systems
        $$\mathcal F_{\mu(j)}/((D\cap N_j)N_{\mu(j)}/N_{\mu(j)})\longrightarrow \mathcal F_{\mu(j)}/((D\cap N_i)N_{\mu(j)}/N_{\mu(j)})$$
        for any $j\geq i \geq i_0$, and by equation~\eqref{eqn upward equality fusion systems} we have $\mathcal F_{\mu(j)}/((D\cap N_i)N_{\mu(j)}/N_{\mu(j)})= \mathcal F_{\mu(i)}/((D\cap N_i)N_{\mu(i)}/N_{\mu(i)})$. So we get morphisms of fusion systems 
        $$
        \varphi_{ij}:\ \mathcal F_{\mu(j)}/((D\cap N_j)N_{\mu(j)}/N_{\mu(j)})\longrightarrow \mathcal F_{\mu(i)}/((D\cap N_i)N_{\mu(i)}/N_{\mu(i)})
        $$
        given by the natural maps (i.e. ``conjugation by $g$'' goes to ``conjugation by $g$'').

        We now know that the inverse limit on the right-hand side of equation~\eqref{eqn block fusion as inverse limit} is well-defined. We still need to show that this inverse limit equals $\Hom_{\mathcal F}(P,Q)$ for any two open subgroup $P,Q\leq D$. That is, we need to check that the natural maps
        \begin{equation}
            \varphi_{i}:\ \Hom_{\mathcal F} (P,Q) \longrightarrow \Hom_{\mathcal F_{\mu(i)}/((D\cap N_i)N_{\mu(i)}/N_{\mu(i)})} (PN_i/N_i,QN_i/N_i)
        \end{equation}
        are well-defined for $i\geq i_0$ and give rise to an isomorphism to the inverse limit.
        
        An element in the domain of $\varphi_i$ is given by conjugation by an element $g$ such that $(P,\hat e_P)^g\leq (Q,\hat e_Q)$. There is a $j\geq i_0$ such that for all $i\geq j$ we have $(PN_i/N_i, e_{P,i})^g \leq (QN_i/N_i, e_{Q,i}$). Hence conjugation by $g$ induces an element of $\Hom_{\mathcal F_{\mu(i)}/((D\cap N_i)N_{\mu(i)}/N_{\mu(i)})} (PN_i/N_i,QN_i/N_i)$ for all $i\geq j$. It will also automatically induce such an element for $i_0\leq i\leq j$ since we saw that the $\varphi_{ij}$ are well-defined. This proves that $\varphi_{i}$ is well-defined for every $i$, and therefore $\Hom_{\mathcal F} (P,Q)$ maps into the inverse limit. 
        
        An element of the inverse limit corresponds to elements $g_i\in G$, one for each $i\geq i_0$, such that $$(P (D\cap N_i)N_{\mu(i)}/N_{\mu(i)}, e_{P(D\cap N_i), \mu(i)})^{g_i}\leq (Q (D\cap N_i)N_{\mu(i)}/N_{\mu(i)}, e_{Q (D\cap N_i), \mu(i)}).$$ 
        There is some $j\geq i_0$ such that for all $i\geq j$ we have $P,Q\supseteq D\cap N_i$. So $P^{g_i}\subseteq QN_{\mu(i)}$. 
       Since $Q$ is a closed $p$-subgroup of $QN_{\mu(i)}$ it is contained in a Sylow $p$-subgroup $R$ of $QN_{\mu(i)}$. If $Q$ was properly contained in $R$, then $Q\cap N_{\mu(i)}$ would also be properly contained in $R\cap N_{\mu(i)}$, contradicting the fact that
         $Q\cap N_{\mu(i)}=D\cap N_{\mu(i)}$ is a Sylow $p$-subgroup of $N_{\mu(i)}$. So $Q=R$ is a Sylow $p$-subgroup of $QN_{\mu(i)}$. In particular there is an $n\in N_{\mu(i)}$ such that $P^{g_in}\subseteq Q$. Without loss of generality we can replace $g_i$ by $g_in$ and assume $P^{g_i}\subseteq Q$ for all $i\geq j$.
        For $i\geq j$ define $$T_i=\left\{ g\in G \ : \ P^g \subseteq Q \textrm{ and } g_ig^{-1}N_{\mu(i)} \in C_{G/N_{\mu(i)}}(PN_{\mu(i)}/N_{\mu(i)}) \right\}.$$ The $T_i$ are closed,  $T_i\supseteq T_{i+1}$ for all $i\geq j$, and we have just shown that they are non-empty. By compactness it follows that their intersection is non-empty, giving us a $g \in G$ such that $P^g\subseteq Q$, $\hat e_{P}^g=\hat e_{P^g}$, and $g$ induces the same homomorphism from $PN_i/N_i$ to $QN_i/N_i$ as $g_i$ for each $i\geq j$. It follows that $\Hom_{\mathcal F} (P,Q)$ surjects onto the inverse limit. 
        
        If two elements $g,h\in G$ induce (by conjugation) the same group homomorphism from $PN_i/N_i$ to $QN_i/N_i$ for all but finitely many $i$, then $gh^{-1}$ induces the identity on $PN_i/N_i$ for all but finitely many $i$, and therefore on $\varprojlim_i PN_i/N_i$. But $\varprojlim_i PN_i/N_i=P$ since $P$ is closed. So the map from $\Hom_{\mathcal F}(P,Q)$ into the inverse limit is injective, and therefore bijective.
   
        The last part of the claim follows since we showed above that, given any $i\geq i_0$, if $P$ and $Q$ contain $D\cap N_i$, then ${\varphi_{i}:\ \Hom_{\mathcal F} (P,Q) \longrightarrow \Hom_{\mathcal F_{\mu(i)}/((D\cap N_i)N_{\mu(i)}/N_{\mu(i)})} (PN_i/N_i,QN_i/N_i)}$ is surjective.
    \end{proof}

    From now on we can think of $\mathcal F_{(D,\hat e)}(G, b)$  as a category whose objects are the \emph{closed} subgroups of $D$, simply by identifying it with the inverse limit in Theorem~\ref{thm fusion as inverse limit}.

    \begin{corollary}
        Let $(D,\hat e)$ be a maximal $(G,b)$-Brauer pair. Then
        $\mathcal F_{(D,\hat e)}(G, b)$ is a pro-saturated pro-fusion system in the sense of Symonds and Stancu (see Definition~\ref{defi profusion system}). We will call $\mathcal F_{(D,\hat e)}(G, b)$  the \emph{block pro-fusion system} of $k\dbl G\dbr b$.
    \end{corollary}

    \begin{prop}\label{prop independence of inverse system}
        Up to conjugacy, the pro-fusion system of $k\dbl G\dbr b$  does not depend on the chain of normal subgroups $(N_i)_{i\in \N}$ chosen in Notation~\ref{notation setup}. 
    \end{prop}
    \begin{proof}
        First note that if $\nu:\ \N \longrightarrow \N$ is a strictly increasing function, then we can turn a Brauer pair with respect to the system $(N_i)_{i\in \N}$ into a Brauer pair for the system $(N_{\nu(i)})_{i\in \N}$ by mapping $\hat e= (e_i)_{i\in \N}$ to $(e_{\nu(i)})_{i\in \N}$. The associated fusion system will be identical, not just conjugate. To see this note that in the proof of Theorem~\ref{thm fusion as inverse limit} the inverse limit $\varprojlim_{i\geq i_0} \mathcal F_{(DN_{\mu(i)}/N_{\mu(i)},e_{\mu(i)})} (G/N_{\mu(i)}, b_{\mu(i)})/((D\cap N_i)N_{\mu(i)}/N_{\mu(i)})$ does not change if we replace $\mu$ by an increasing function $\mu'$ such that $\mu'(i)\geq \mu(i)$ for all $i$. So without loss of generality, $\mu$ takes values in the image of $\nu$. But then the corresponding inverse limit for the subsystem $(N_{\nu(i)})_{i\in \N}$ is just the inverse limit indexed by a cofinal subsystem of the original one, and therefore is the same. 

        Let us now assume that we are given another system $(M_i)_{i\in \N}$. Since $\bigcap _{i\in \N} M_i= \bigcap_{i\in\N} N_{i}=1$ we can find strictly increasing functions $\alpha,\beta:\ \N \longrightarrow \N$ such that
        \begin{equation}
            N_{\alpha(1)} \geq M_{\beta(1)}\geq N_{\alpha(2)}\geq M_{\beta(2)} \geq \ldots.
        \end{equation}
        That is, we get a system $(L_i)_{i\in \N}$ such that $L_{2i-1}=N_{\alpha(i)}$ and $L_{2i}=M_{\beta(i)}$. But then the previous paragraph shows that all three systems lead to conjugate block fusion systems (here we need to conjugate since we can only produce a Brauer pair for a subsystem from a Brauer pair for a bigger system, not the other way around).
    \end{proof}

    In a sense we can think of $\mathcal F_{(D,\hat e)}(G, b)$ as the smallest pro-fusion system such that all but finitely many of the fusion systems of the finite quotients of $k\dbl G\dbr b$ are contained in the appropriate quotient of $\mathcal F_{(D,\hat e)}(G, b)$.

    \begin{prop}\label{prop fusion embedding}
        Let $(D,\hat e)$ be a maximal $(G,b)$-Brauer pair. Then there is an open normal subgroup $N_0$ of $G$ such that we have an embedding of fusion systems 
        \begin{equation}
            \mathcal F_{(DN/N, e_N)} (G/N, b_N) \hookrightarrow \mathcal F_{(D,\hat e)}(G, b)/(D\cap N)
        \end{equation}
        whenever $N$ is an open normal subgroup of $G$ contained in $N_0$, the block idempotent $b_N$ is as in Remark~\ref{remark blocks} and $(DN/N, e_N)$ is some maximal $(G/N, b_N)$-Brauer pair.
    \end{prop}
    \begin{proof}
        We pick our $N_0=N_{i_0}$ with $i_0$ as in the proof of Theorem~\ref{thm fusion as inverse limit}. In light of Proposition~\ref{prop independence of inverse system} we can assume that an $N$ as in the assertion is equal to $N_i$ for some $i\geq i_0$, and $e_i$ is conjugate to $e_N$. Now the claim follows from Theorem~\ref{thm fusion as inverse limit}, since $$\mathcal F_{(D,\hat e)}(G, b)/(D\cap N_i)=\mathcal F_{(DN_{\mu(i)}/N_{\mu(i)},e_{\mu(i)})} (G/N_{\mu(i)}, b_{\mu(i)})/((D\cap N_i)N_{\mu(i)}/N_{\mu(i)}),$$ and $\mathcal F_{(DN_i/N_i, e_i)}(G/N_i, b_i)$ is a subsystem of the right-hand side, as was seen in the part of the proof of Theorem~\ref{thm fusion as inverse limit} where $\mu$ was constructed.
      \end{proof}

      \section{Nilpotent blocks}

    One of the strongest applications of fusion systems in the block theory of finite groups is Puig's theory of nilpotent blocks. Recall that a block of a finite group with defect group $D$ is called \emph{nilpotent} if the associated fusion system is trivial in the sense that it is equal to $\mathcal F_D(D)$. One can define $\mathcal F_D(D)$ for a  pro-$p$ group $D$ in the same way as for finite $p$-groups \cite{StancuSymonds}.

    \begin{defi}
        Let $G$ be a countably based profinite group, let $b\in Z(k\dbl G\dbr )$ be a block idempotent and let $(D,\hat e)$ denote a maximal $(G,b)$-Brauer pair. We call the block $k\dbl G\dbr b$ \emph{nilpotent} if $\mathcal F_{(D,\hat e)}(G,b)=\mathcal F_D(D)$.
    \end{defi}

    Note that while we are asking for $G$ to be countably based, by Corollary~\ref{corol can take G countably based} this really should be seen as a restriction on the defect group $D$ rather than as a restriction on $G$.

    \begin{thm}\label{thm puig in body}
        Let $G$ be a countably based profinite group and let $k\dbl G\dbr b$ be a nilpotent block with defect group $D$. If $D$ is topologically finitely generated, then $k\dbl G\dbr b$ is Morita equivalent to $k\dbl D\dbr $.
    \end{thm}
    \begin{proof}
        By Proposition~\ref{prop fusion embedding}, $k\dbl G\dbr b$ can be written as an inverse limit of nilpotent blocks of finite groups, with defect groups $DN_i/N_i$, where the $N_i$ are the open normal subgroups of $G$ from Notation~\ref{notation setup} and $i\geq i_0$ for some $i_0\in\N$. By Proposition~\ref{prop inverse limit basic}, the block $k\dbl G\dbr b$ is then Morita equivalent to the inverse limit $A=\varprojlim_{i\geq i_0} A_{N_i}$ of a surjective inverse system, where each $A_{N_i}$ is the basic algebra of a nilpotent block with defect group $DN_i/N_i$. By Puig's structure theory (see \cite[Theorem 8.11.5]{LinckelmannVolII}) the algebra $A_{N_i}$ is isomorphic to $kDN_i/N_i$. Since $D$ is topologically finitely generated, the quotient $D/\Phi(D)$, where $\Phi(D)=D^p\cdot [D,D]$ is the Frattini subgroup, is finite. Note that $\dim A_{N_i}/J^2(A_{N_i}) \leq 1 + |D/\Phi(D)|$ for any $N_i$, which implies $\dim A/J^2(A)\leq 1 + |D/\Phi(D)|<\infty$. As mentioned at the beginning of
        Section~\ref{section completed path algebras} this means that we can write $A$ as a quotient of the completed path algebra of a finite quiver.
        To do this, pick $g_1,\ldots, g_n\in D$ such that the images of $1-g_i$ generate $D/\Phi(D)$, let $Q$ be a bouquet of $n$ loops, and let $\varphi:\ k\dbl Q\dbr  \longrightarrow k\dbl D\dbr $ denote the map sending the loops to $1-g_1,\ldots, 1-g_n$. Then $\varphi$ is surjective, since its composition with the natural epimorphism $\nu_{N_i}:\ k\dbl D\dbr \twoheadrightarrow kDN_i/N_i$ is surjective for any $i\geq i_0$ due to the fact that the images of the $1-g_j$ span $J(kDN_i/N_i)/J^2(kDN_i/N_i)$.  If we define $I_i=\Ker(\nu_{N_i}\circ\varphi)$ then clearly $k\dbl Q\dbr /\bigcap_i I_i \cong k\dbl D\dbr $ and $k\dbl Q\dbr /I_i \cong kDN_i/N_i \cong A_{N_i}$. By Proposition~\ref{prop unique limit} it follows that $\varprojlim_{i\geq i_0} A_{N_i}\cong k\dbl D\dbr $. 
    \end{proof}

    \section{Blocks of dihedral defect}

    It is a well-known consequence of the theory of nilpotent blocks that all blocks with defect group $C_{2^n}$ are Morita equivalent to $kC_{2^n}$ -- simply because $C_{2^n}$ does not allow any non-trivial fusion systems to be defined on it. In the profinite case, we get a similar result for blocks of infinite dihedral defect $D_{2^\infty}$.

    \begin{prop}
        Let $\mathcal F$ be a pro-saturated pro-fusion system on $D_{2^\infty}$. Then $\mathcal F = \mathcal F_{D_{2^\infty}}(D_{2^\infty})$.
    \end{prop}
    \begin{proof}
        By definition, $\mathcal F=\varprojlim_{i\in I} \mathcal F_i$ for saturated fusion systems $\mathcal F_i$ on finite quotients of $D_{2^\infty}$. Here $I$ denotes some directed indexing set. By \cite[Lemma 4.2]{StancuSymonds} we can assume that the inverse system is surjective. We can write $D_{2^\infty}=\overline{\langle a,b \ : \ b^2, \ baba \rangle}$, where the bar denotes the pro-$2$ completion.
         Note that all normal subgroups of $D_{2^\infty}$ of index greater than two are of the form $\overline{\langle a^{2^n} \rangle}$ for $n\geq 1$, and therefore leave quotient $D_{2^{n+1}}$. Hence we can find, for any $n_0\in \N$, elements $j<i \in I$ such that $\mathcal F_i$ is a fusion system on $D_{2^n}$, $\mathcal F_j$ is a fusion system on $D_{2^m}$ for $n > m \geq n_0$, and the map $D_{2^n} \twoheadrightarrow D_{2^m}$ is (without loss of generality) the natural epimorphism. 

        We will show that $\mathcal F_j=\mathcal F_{D_{2^m}}(D_{2^m})$ provided $n_0\geq 3$. Since this holds for all $j$ except those where $\mathcal F_j$ is defined on a group of order $\leq 4$, it will follow that $\mathcal F$ is trivial.
        By Alperin's Fusion Theorem it follows \cite[Corollary 8.2.9]{LinckelmannVolII} that $\mathcal F_j$ is trivial if and only if $\Aut_{\mathcal F_j}(P)$ is a $2$-group for all subgroups $P\leq D_{2^m}$. But all subgroups of $D_{2^m}$ are either cyclic or dihedral. A cyclic group of $2$-power order has an automorphism group of $2$-power order, as does a dihedral group of $2$-power order $\geq 8$. The only subgroups of $D_{2^m}$ for which $\Aut_{\mathcal F_j}(P)$ might not be a $2$-group are the Klein four subgroups of $D_{2^m}$, of which there are two conjugacy classes, represented by $V_{m,1}=\langle a^{2^{m-1}}, b\rangle$ and $V_{m,2}=\langle a^{2^{m-1}}, ab\rangle$. But their preimages $W_{m,1}$ and $W_{m,2}$ in $D_{2^n}$ are dihedral groups of order $2^{n-m+2}$, whose automorphism groups are $2$-groups. Hence, for $s\in \{1,2\}$, the image of $\Aut_{\mathcal F_i}(W_{m,s})$ in $\Aut_{\mathcal F_j}(V_{m,s})$ is a $2$-group, which by our surjectivity assumption on the inverse system implies that  $\Aut_{\mathcal F_j}(V_{m,s})$ is a $2$-group. Hence $\mathcal F_j$ is trivial for all $j$ sufficiently large, and therefore so is $\mathcal F$.
    \end{proof}

    The above proposition combined with Theorem~\ref{thm puig in body} immediately implies the corollary below, which classifies the blocks with defect group $D_{2^\infty}$ up to Morita equivalence. Note that we do not need to ask for $G$ to be countably based due to Corollary~\ref{corol can take G countably based}. 

    \begin{corollary}\label{Corollary Dinfty block is kDinfty in body}
        Let $G$ be a profinite group and let $k\dbl G\dbr b$ be a block with defect group isomorphic to $D_{2^\infty}$. Then $k\dbl G\dbr b$ is Morita equivalent to $k\dbl D_{2^\infty}\dbr $.
    \end{corollary}

\section{Inverse limits of tame blocks}\label{section limits of tame blocks}

We have shown using group theoretic methods that there is only one Morita equivalence class of block with defect group $D_{2^{\infty}}$, and it appears that this classification cannot be obtained using purely algebra-theoretic methods and the corresponding classification of finite blocks, as was done with the infinite cyclic defect group.  However, the class of algebras that are inverse limits of blocks with finite dihedral defect group is remarkably small, and the algebras obtained are very simple.  We present the classification without proof.

\begin{prop}
            Let $k$ be an algebraically closed field of characteristic $2$.  Let $B$ be the inverse limit of an inverse system of blocks $B_{n}$, where $B_n$ is a block of a finite group with finite dihedral defect group.  Then $B$ is Morita equivalent to the bounded completed path algebra $k\db{Q}/I$, where 
            either:
$$Q = \begin{tikzcd}
	\bullet \arrow[out=405,in=-15,loop,"a"]
	\arrow[out=140,in=-160,swap,loop,"b"]
\end{tikzcd}\,\hbox{ and }\,I = \langle a^2, b^2\rangle;$$
$$Q = \begin{tikzcd}
			\bullet \arrow[r,bend left,"b_1"] 
			\arrow[out=140,in=-160,swap,loop,"a"]& \bullet \arrow[l,bend left,"b_2"] 
		\end{tikzcd}\,\hbox{ and }\,I = \langle b_1b_2, a^2\rangle;$$
$$Q = \begin{tikzcd}
			\bullet \arrow[r,bend left,"b_{2}"] 
			& \bullet \arrow[l,bend left,"b_{1}"] \arrow[r,bend right,swap,"a_{1}"] & \bullet \arrow[l,bend right,swap,"a_{2}"] 
		\end{tikzcd}\,\hbox{ and }\,I = \langle a_1a_2, b_1b_2\rangle.$$
\end{prop}

\section{Alternative definitions of Brauer pairs and open questions}\label{section Brauer pair naive}

Our treatment of Brauer pairs is rather delicate, for the following reason.  One would of course like to study Brauer pairs in terms of Brauer pairs of finite quotient groups of $G$.  But if $N$ is an open normal subgroup of $G$, the natural projection $G\to G/N$ induces a surjective map $C_G(Q)\to C_G(Q)N/N$, whereas the finite theory applies to the potentially larger subgroup $C_{G/N}(QN/N)$.  One has a map
$C_G(Q) \onto C_G(Q)N/N\hookrightarrow C_{G/N}(QN/N)$, but it need not be surjective, and thus one must be careful when restricting to centres. As a consequence, Definition~\ref{defi brauer pair profinite} does not match the naive generalization of the definition of Brauer pairs for finite groups.

\begin{question}\label{question naive brauer pairs}
    Is it possible to construct block pro-fusion systems by defining Brauer pairs for profinite groups~$G$ as pairs $(P,e)$, where $P$ is a closed $p$-subgroup of $G$ and $e$ is a primitive idempotent in $Z(k\db{C_G(P)})$, with the relation ``$\leq$'' defined using the Brauer homomorphism for profinite groups?
\end{question}

It is not at all clear whether Brauer pairs defined in this way have the necessary properties to define a category on the defect group of a block, and, assuming they do, whether that category would turn out to be a pro-fusion system. Furthermore, there is currently no axiomatic characterization of pro-fusion systems, so an analogue of Theorem~\ref{thm fusion as inverse limit} would still be required. Nevertheless, a positive answer to Question~\ref{question naive brauer pairs} could help answer the following obvious question:

\begin{question}
    Is it possible to extend the definiton of block pro-fusion systems to blocks whose defect groups are not countably based?
\end{question}

A positive answer would likely require further results on block fusion systems for finite groups along the lines of Lemmas~\ref{lemma nu minus}~and~\ref{lemma fusion system grows}. But it is not even clear if we should expect the answer to be affirmative.

To finish let us prove one proposition which indicates that Question~\ref{question naive brauer pairs} is reasonable.

\begin{prop}
Let $G$ be countably based and let $P$ be an open subgroup of a Sylow $p$-subgroup of $G$. Then there is a bijection
  \begin{equation}
      \{ \textrm{ elements $\hat e$ as in Definition~\ref{defi brauer pair profinite} } \} \longleftrightarrow \{ \textrm{ primitive idempotents in $Z(k\db{C_G(P)})$ }\}.
  \end{equation} 
\end{prop}
\begin{proof}
  We use Notation~\ref{notation setup}. Using Proposition~\ref{prop independence of inverse system} we can replace the $N_i$ by a cofinal subsystem such that the following hold for all $i\in\N$:
  \begin{enumerate}
      \item $N_i\cap P$ is a Sylow $p$-subgroup of $N_i$, and
      \item $C_{G/N_{i+1}}(PN_{i+1}/N_{i+1}) \twoheadrightarrow C_G(P)N_i/N_i$.
  \end{enumerate}
  The second condition is satisfiable since $C_G(P)=\varprojlim_{i\in \N} C_{G/N_{i}}(PN_{i}/N_{i})$. With these definitions we get a diagram of group algebras
  \begin{equation}\label{eqn diagramme}
      \xymatrix{
          \cdots \ar[r]& kC_{G/N_{i+1}}(PN_{i+1}/N_{i+1})\ar@{->}[r]^{\nu_i} \ar@{->>}[d]^{\varphi_i} & kC_{G/N_{i}}(PN_{i}/N_{i}) \ar@{->>}[d]^{\varphi_{i-1}} \ar[r] & \cdots\\
          \cdots\ar@{->>}[r] & kC_G(P)N_i/N_i \ar@{->>}[r]^{\nu_i} \ar@{^{(}->}[ur]^{\iota_i} &  kC_G(P)N_{i-1}/N_{i-1} \ar@{->>}[r] &\cdots
      }
  \end{equation}
  where every triangle commutes and the maps are the natural ones. Let $\bar K_i$ be the kernel of the group homomorphism $C_{G/N_{i+1}}(PN_{i+1}/N_{i+1}) \twoheadrightarrow C_G(P)N_i/N_i$. We have $\bar K_i\leq N_{i}/N_{i+1}$. 
  Since $N_i\cap P$ is a Sylow $p$-subgroup of $N_i$, the group $(N_i\cap P)N_{i+1}/N_{i+1}$ is a Sylow $p$-subgroup of $N_i/N_{i+1}$. Since $\bar K_i$ centralizes  $(N_i\cap P)N_{i+1}/N_{i+1}$, any Sylow $p$-subgroup $\bar Q_i$ of $\bar K_i$ is contained in $(N_i\cap P)N_{i+1}/N_{i+1}$ (as otherwise the product of the two $p$-groups is a bigger $p$-subgroup of $N_i/N_{i+1}$). In particular $\bar Q_i$ is central and therefore normal in $C_{G/N_{i+1}}(PN_{i+1}/N_{i+1})$. Hence we have epimorphisms
  \begin{equation}
      kC_{G/N_{i+1}}(PN_{i+1}/N_{i+1}) \twoheadrightarrow  kC_{G/N_{i+1}}(PN_{i+1}/N_{i+1})/\bar Q_i \twoheadrightarrow kC_{G/N_{i+1}}(PN_{i+1}/N_{i+1})/\bar K_i\cong kC_G(P)N_i/N_i.
  \end{equation}
  The first epimorphism corresponds to a quotient by a central $p$-subgroup, and therefore induces a bijection of block idempotents by \cite[Theorem~8.11]{NagaoTsushima}. The second epimorphism may send some block idempotents to zero, but since it corresponds to a quotient by a $p'$-group it will induce a bijection on those block idempotents that it does not send to zero by \cite[Theorem~8.8]{NagaoTsushima}. We conclude that the vertical maps $\varphi_i$ in diagram~\eqref{eqn diagramme} send block idempotents either to block idempotents or to zero.

  Recall from Lemma~\ref{lemma nu minus} that if we set $C=C_{PN_{i+1}/N_{i+1}, N_i/N_{i+1}}$ then $\nu_i$ restricts to a map $$Z(kC_{G/N_{i+1}}(PN_{i+1}/N_{i+1}))^{C}\longrightarrow Z(kC_{G/N_{i}}(PN_{i}/N_{i})).$$ By definition, $CN_i/N_i\leq C_{G/N_{i}}(PN_{i}/N_{i})$, and therefore $C\leq C_G(P)N_{i-1}/N_{i+1}$ by well-definedness of $\varphi_{i-1}$. A primitive idempotent $e$ in $Z(kC_{G/N_{i+1}}(PN_{i+1}/N_{i+1}))^{C}$ is therefore a sum of primitive idempotents $e_1,\ldots, e_r$ ($r\in \N$) in $Z(kC_{G/N_{i+1}}(PN_{i+1}/N_{i+1}))$ conjugate by elements of $N_{i-1}$. The elements $\varphi_i(e_1),\ldots, \varphi_i(e_r)$ will also be conjugate by elements of $N_{i-1}$, which implies that $\nu_i(\varphi_i(e_1))=\ldots=\nu_i(\varphi_i(e_r))$. By orthogonality of the $e_j$, this implies that either $r=1$ or $\nu_i(\varphi_i(e_1))=\ldots=\nu_i(\varphi_i(e_r))=0$. In the first case $e$ is actually a block idempotent itself and, in the notation of Lemma~\ref{lemma nu minus}, $e=\nu_{PN_i/N_i}^-(f)$ for every block idempotent $f$ in $Z(kC_{G/N_{i}}(PN_{i}/N_{i}))$ with $\nu_i(e)f\neq 0$. In the second case we have $\varphi_{i-1}(\nu_i(e))=0$, and therefore $\varphi_{i-1}(f)=0$ for any block idempotent $f$ in $Z(kC_{G/N_{i}}(PN_{i}/N_{i}))$ with  $e=\nu_{PN_i/N_i}^-(f)$.

  Now take a primitive idempotent $f\in Z(k\db{C_G(P)})$. Such an $f$ corresponds to the equivalence class of a family $(f_i)_{i\geq i_0}$ (where $i_0\in\N$), where $f_i\in Z(kC_G(P)N_i/N_i)$ is a primitive idempotent and $\nu_i(f_i)f_{i-1}\neq 0$ for all $i\geq i_0$ (see \cite[Remark~4.4]{FranquizMacQuarrieBrauer}). By the previous two paragraphs there are unique primitive idempotents $e_i\in Z(kC_{G/N_{i+1}}(PN_{i+1}/N_{i+1}))$ such that $\varphi_i(e_i)=f_i$. Note that for $i\geq i_0+1$ the idempotent $e_{i-1}$ is $C$-invariant with $C$ as above, and therefore so is $f_{i-1}$. The idempotent $f_i$ is uniquely characterised by the condition $\nu_i(f_i)f_{i-1}\neq 0$, so $f_i$ is $C$-invariant as well, and therefore so is $e_i$. 
  By the commutativity of diagram~\eqref{eqn diagramme} we have $\nu_i(e_i)e_{i-1}\neq 0$, which means $e_i=\nu_{PN_i/N_i}^-(e_{i-1})$. In particular $\hat e = (e_i)_{i\in \N}$ satisfies the conditions of Definition~\ref{defi brauer pair profinite} (technically we have shown something stronger, but we had to thin out the system of normal subgroups $N_i$ for this).

  Now take $\hat e = (e_i)_{i\in \N}$ as in Definition~\ref{defi brauer pair profinite}.
  Then $e_i\nu_{PN_i/N_i}^-(e_{i-1})\neq 0$ for all $i \geq i_0$, where $i_0$ is chosen sufficiently large. 
  This means that $\nu_{PN_i/N_i}^-(e_{i-1})$ is the $C$-orbit sum of $e_i$. In particular $\nu_i(e_i)e_{i-1}\neq 0$ for all $i\geq i_0$, since by $C$-equivariance of $\nu_i$ having $\nu_i(e_i)e_{i-1}= 0$ would imply $\nu_i(\nu_{PN_i/N_i}^-(e_{i-1})) e_{i-1}=0$, which directly contradicts the definition of $\nu_{PN_i/N_i}^-(e_{i-1})$.
  Now set $f_i=\varphi_i(e_i)$ and consider the family $(f_i)_{i\geq i_0}$. The fact that $\nu_i(e_i)e_{i-1}\neq 0$ implies $\nu_i(e_i)=\iota_i(\varphi_i(e_i))\neq 0 $, so $\varphi_i(e_i)=f_i\neq 0$. It follows that all $f_i$ for $i\geq i_0$ are block idempotents, as required.
  It remains to show that $\nu_i(f_i)f_{i-1}\neq 0$ for all $i$ sufficiently large, say $i\geq i_0+1$. Define $f_i'$ to be the unique block idempotent of $kC_G(P)N_i/N_i$ with $\nu_i(f'_i)f_{i-1}\neq 0$, and let $e_i'$ denote the unique block idempotent of $ kC_{G/N_{i+1}}(PN_{i+1}/N_{i+1})$ with $\varphi_i(e_i')=f_i'$. Clearly $\nu_i(e_i')e_{i-1}\neq 0$. Note that $e_{i-1}$ is $C$-invariant, so by uniqueness the same is true for $f_{i-1}$, $f_i'$ and $e_i'$, that is, $e_i'\in  Z(kC_{G/N_{i+1}}(PN_{i+1}/N_{i+1}))^C$. It follows that $e_i'=\nu_{PN_i/N_i}^-(e_{i-1})$. So $e_i=e_i'$ and therefore ${\nu_i(f_i)f_{i-1}\neq 0}$.
\end{proof}

\bibliographystyle{abbrv}
\bibliography{refpap}
\end{document}